\documentclass[12pt,a4paper]{article}
\usepackage{amsmath} 
\usepackage{amsthm}
\usepackage{amsfonts}
\usepackage{amssymb}
\usepackage{graphicx}
\usepackage{subfigure}
\usepackage{setspace}
\usepackage{color}
\usepackage{cite}
\usepackage{fourier}
\usepackage{multirow}
\usepackage[left=2.75cm,right=2.75cm,top=2.5cm,bottom=2.5cm,includeheadfoot]{geometry}

\usepackage[authoryear]{natbib}
\allowdisplaybreaks

\newtheorem{prop}{Proposition}[section]
\newtheorem{thm}{Theorem}[section]
\newtheorem{lemma}{Lemma}[section]

\newtheorem{claim}{Claim}
\theoremstyle{remark}
\newtheorem{remark}{Remark}[section]

\newcommand{\overbar}[1]{\mkern 1.7mu\overline{\mkern-1.7mu#1\mkern-1.7mu}\mkern 1.7mu}

\usepackage{enumerate}
\usepackage{hyperref}

\def\R{\mathbb R}
\def\E{\mathbb E}
\def\Z{\mathbb Z}

\def\o{\omega}

\def\wtilde{\widetilde}
\def\cm{\operatorname{cum}}

\usepackage[textsize=scriptsize,backgroundcolor=yellow,linecolor=yellow]{todonotes}
\begin{document}

\title{A simple test for white noise in functional time series}
\author{Pramita Bagchi, Vaidotas Characiejus,  Holger Dette \\
Ruhr-Universit\"at Bochum \\
Fakult\"at f\"ur Mathematik \\
44780 Bochum \\
Germany}

\maketitle
\begin{abstract}
We propose a new procedure for white noise testing of a functional time series. Our approach is based on an explicit representation  of the  $L^2$-distance between the spectral density operator and its best  ($L^2$-)approximation by a spectral density operator corresponding to a white noise process. The estimation of this distance can be easily accomplished by sums of periodogram kernels  and it is shown that an appropriately standardized version of the estimator is asymptotically normal distributed under the null hypothesis (of functional white noise) and under the alternative. As a consequence we obtain a very simple test (using the quantiles of the normal distribution) for the hypothesis of a white noise functional process. In particular the test does neither require the estimation of a long run variance (including a fourth order cumulant) nor resampling procedures to calculate critical values. Moreover, in contrast to all other methods proposed in the literature our approach also allows to test for ``relevant'' deviations from white noise and to construct confidence intervals for a measure which measures the discrepancy of the underlying process from a functional white noise process.
\end{abstract}


\section{Introduction}\label{sec1} 
\def\theequation{1.\arabic{equation}}
\setcounter{equation}{0}
The problem of testing for white noise in dependent data is of particular importance because these tests are frequently used to check the adequacy of a postulated parametric model. The seminal work on this problem can be found in the papers of \cite{boxpie1970,pierce1972,ljubox1978} who proposed  \emph{portmanteau tests} to check the goodness of fit  of an ARMA model. They operate in the  time domain and are based on a sum of squared correlations with fixed lag truncation number [see also \cite{mokkadem1997,dettspre2000} for some more recent references]. The asymptotic properties of the different test statistics considered in these papers are usually derived under the assumption of independent and identically distributed (i.i.d.) innovations and several authors point out that these tests are are not reliable if the innovations are uncorrelated but not independent [see \cite{romtho1996} and \cite{fraroyzak2005} among others].

An alternative to the tests operating in the time domain are frequency domain tests, which are based on a comparison between the spectral density corresponding to the process of the innovations and the spectral density of a white noise. For example, \cite{hong1996} proposed to use an $L^2$-distance between a kernel-based spectral density estimator and the spectral density of the noise under the null hypothesis to construct a test statistic and this approach has been more recently further developed by  \cite{shao2011}. We also refer to \cite{dettspre2003,paparoditis2000} for some  results testing more general hypotheses by investigating distances between a parametric and non-parametric spectral density estimate. Other authors propose to use  normalized cumulated deviations between a non-parametrically  and a parametrically estimated spectral density [see for example \cite{deo2000}]. All  methods mentioned in this and the previous paragraph require the specification of a regularization parameter (lag number or bandwidth). \cite{Dette_Kinsvater_Vetter_2011} proposed to estimate the $L^2$-distance between the unknown  density directly  using sums of (squared) periodograms. The corresponding test statistic does not require regularization and under the null hypothesis and the additional assumption of a linear moving averaging process with Gaussian innovations its asymptotic distribution  is a centered  normal distribution with an easily estimable variance. As a consequence a very simple test for white noise can be proposed with attractive finite sample properties. 

Due to the increasing demand in analyzing data providing information about curves, surfaces or anything else varying over a continuum many of these methods have been recently further developed to be applicable for functional data.  For a general review on Functional  data analysis (FDA) with dependent observations we refer the interested readers to the monograph by \cite{horvkoko2012}. A test for the hypotheses of white  noise 
of a sequence of functional observations in the time domain has been  proposed  by \cite{gabrkoko2007}. They combine principle components for functional data analysis with a ``classical'' portmanteau test. More recently, \cite{horhusric2013}  considered an alternative  portmanteau test which is  based on the sum of the $L^2$-norms of the empirical covariance kernels. As the validity of these tests is only justified under the i.i.d.\ assumption of the innovation process (and therefore not robust to  dependent white noise), \cite{zhang2016} proposed a spectral domain test using a cumulative distance between the periodogram function and its integral with respect to the frequency [for an early result in the one-dimensional case we also refer to \cite{dahlhaus1988}]. This author proved weak convergence of an appropriately standardized version of this process and derived a Cramer von Mises type statistic with non-pivotal limiting null distribution. To solve this problem a bootstrap procedure is introduced to generate critical values. 

The present paper is devoted to an alternative test in the spectral domain for the hypothesis of white noise functional data. Our approach is based on a direct estimate  of the $L^2$-distance between (unknown) spectral density operator and its best approximation by an operator corresponding to functional white noise process. This distance can be estimated directly by sums of periodogram kernels (thus we do not estimate the spectral density kernel, but just real valued functionals of it). We show that  that the corresponding test statistic is asymptotically normal distributed such that critical values can easily be obtained. The main advantage of our approach is its simplicity as it neither requires regularization nor bootstrap in its implementation. In particular the latter fact makes it  computationally very efficient for functional data. Moreover, we also demonstrate by means of a simulation study that the new test is very competitive to an alternative procedure which has recently been proposed in the literature [see \cite{zhang2016}].

The corresponding model is introduced in \autoref{sec2}. \autoref{sec3} is devoted to the new distance, its estimate and the corresponding asymptotic theory. We also note that our approach (as it is based on a distance) provides a measure of deviation from a functional white noise for which we provide an explicit (and simple) confidence interval. Other statistical applications are also discussed in this section. In \autoref{sec4} we investigate the finite sample properties of the new test and compare it with the  alternative test proposed by \cite{zhang2016}. Finally, the proofs of the main results are given in \autoref{sec5} and \autoref{sec6}.

\section{Notations and preliminaries} \label{sec2} 
\def\theequation{2.\arabic{equation}}
\setcounter{equation}{0}
Let $L^p([0,1]^k,\mathbb C)$ with $p\ge1$ and $k\ge1$ denote the Banach space of measurable functions $f:[0,1]^k\to\mathbb C$ whose absolute value raised to the $p$-th power has finite integral. The norm of $L^p([0,1]^k,\mathbb C)$ is defined by
\[
	\|f\|_p\equiv\bigg(\int_{[0,1]^k}|f(x)|^p\;dx\biggr)^{1/p}<\infty.
\]
Note that the equality of the $L^p([0,1]^k,\mathbb C)$ elements is understood in the sense of the norm of their difference being zero. The real and the imaginary parts of the complex number $x$ are denoted by $\operatorname{Re}x$ and $\operatorname{Im}x$ respectively. $\overline x$ denotes the complex conjugate of $x\in\mathbb C$ and $i$ is the imaginary unit, i.e. $i=\sqrt{-1}$. $L^2([0,1]^k,\mathbb C)$ is also a Hilbert space with the inner product given by 
\[
	\langle f,g\rangle\equiv\int_{[0,1]^k}f(x)\overline{g(x)}dx
\]
for $f,g\in L^2([0,1]^k,\mathbb C)$. $L^p([0,1]^k,\mathbb R)$ denotes the corresponding space of real-valued functions.

Suppose that $\{X_t\}_{t \in \Z}$ is a functional time series such that $X_t$ is a random element of $L^2([0,1],\R)$ for each $t\in\mathbb Z$. We assume that $\{X_t\}_{t\in\mathbb Z}$ is strictly stationary in the sense that for any finite set of indices $I\subset\mathbb Z$ and any $s\in\mathbb Z$, the joint law of $\{X_t,t\in I\}$ coincides with that of $\{X_{t+s},t\in I\}$.  If $\E\|X_0\|_2<\infty$, there exists a unique function $\mu\in L^2([0,1],\mathbb R)$ such that $\E\langle f,X_0\rangle=\langle f,\mu\rangle$ for any $f\in L^2([0,1],\mathbb R)$. It follows that $\mu(\tau)=\E X_0(\tau)$ for almost all $\tau\in[0,1]$. For all $t,s\in\mathbb Z$ and $\tau,\sigma\in[0,1]$, we define the autocovariance kernel $r_t\in L^2([0,1]^2,\mathbb R)$ at lag $t\in\mathbb Z$ as
\begin{equation}\label{eq:autocovker}
	r_t(\tau,\sigma)
	=\E[(X_{t+s}(\tau)-\mu(\tau))(X_s(\sigma)-\mu(\sigma))]
\end{equation}
provided that $\mathbb E\|X_0\|_2^2<\infty$. The autocovariance operator $\mathcal R_t:L^2([0,1],\R) \to L^2([0,1],\R)$ at lag $t$ is defined as the integral operator induced by the kernel $r_t$, i.e.,
\[
	\mathcal{R}_th(\tau) = \int_0^1 r_t(\tau,\sigma)h(\sigma)d\sigma
\]
for each $h\in L^2([0,1],\mathbb R)$ and $\tau\in[0,1]$.

\medskip
The joint cumulant of a real- or complex-valued random vector $(\begin{array}{ccc}\xi_1 & \ldots & \xi_n\end{array})^{\mathrm T}$ with $\operatorname E|\xi_i|^n<\infty$ for $i=1,\ldots,n$ is denoted by $\operatorname{cum}(\xi_1,\ldots,\xi_n)$. The cumulant of order $p\ge1$ of a real- or complex-valued random variable $\xi$ with $\operatorname E|\xi|^p<\infty$ is denoted by $\operatorname{cum}_p(\xi)$. We quantify the dependence among the observations $\{X_t\}_{t\in\mathbb Z}$ using the cumulant kernel of the series. The pointwise definition of the $k$-th order cumulant kernel is given by
$$
	\cm\left(X_{t_1}(\tau_1),\dots,X_{t_k}(\tau_k)\right)
	=\sum_{\nu = \nu_1\cup\dots\cup\nu_p} (-1)^{p-1}(p-1)!\prod_{l=1}^p \E\Big[\prod_{j \in \nu_l} X_{t_j}(\tau_j) \Big],
$$
where the sum extends over all unordered partitions of $\{1,2,\dots,k\}$. The cumulant kernel of order $k$ is an element of $L^2([0,1]^k,\mathbb R)$ under the assumption of $\E\| X_0 \|_2^k < \infty$. As in~\cite{Panaretos_Tavakoli_2013_Supp}, we introduce the cumulant spectral density of order $k$ defined by
\begin{multline}\label{eq:cum_spectra}
	f_{\omega_1,\dots,\omega_{k-1}}(\tau_1,\dots,\tau_k)\\
 	=\frac1{(2\pi)^{k-1}}\sum_{t_1,\dots,t_{k-1} = -\infty}^{\infty}\exp\Big(-i\sum_{j=1}^{k-1}\omega_j t_j\Big)\cm\big(X_{t_1}(\tau_1),\dots,X_{t_{k-1}}(\tau_{k-1}),X_0(\tau_k)\big),
\end{multline}
where the series converges in $L^2$ under the cumulant mixing condition

\medskip
~~~ (B) ~ $\displaystyle\sum_{t_1,\dots,t_{k-1} = -\infty}^{\infty} \left\|\cm\big(X_{t_1},\dots,X_{t_{k-1}},X_0\big)\right\|_2 < \infty.$

The cumulant spectral density of order $k$ is uniformly bounded in $\omega_1,\dots,\omega_{k-1}$. 

\medskip
Next we introduce some notations for operators. Let $H_1$ and $H_2$ be two separable Hilbert spaces. For any operator $A$ from $H_1$ to $H_2$, the Hermitian adjoint of A is denoted by $A^*$. A bounded linear operator $A:H_1\to H_2$ is a Hilbert-Schmidt operator if
\[
	|\mkern-1.5mu|\mkern-1.5mu|A|\mkern-1.5mu|\mkern-1.5mu|_2^2
    \equiv\sum_{i=1}^\infty\|Ae_i\|^2<\infty,
\]
where $\|\cdot\|$ is the norm of the space $H_2$ and $\{e_i\}_{i\ge1}$ is any orthonormal basis of $H_1$. The space of the Hilbert-Schmidt operators is also a Hilbert space with the inner product defined by
\[
	\langle A,B\rangle_{\mathrm HS}
    \equiv\sum_{i=1}^\infty\langle Ae_i,Be_i\rangle
\]
for two Hilbert-Schmidt operators $A$ and $B$, where $\langle\cdot,\cdot\rangle$ is the inner product of the space $H_2$. Again, this definition is independent of the choice of the basis $\{e_i\}_{i\ge1}$. A bounded linear operator $A:L^2([0,1]^k,\mathbb C)\to L^2([0,1]^k,\mathbb C)$ is a Hilbert-Schmidt operator if and only if there exists a kernel $k_A\in L^2([0,1]^{2k},\mathbb C)$ such that
\[
	Af(x)=\int_{[0,1]^k}k_A(x,y)f(y)dy
\]
almost everywhere in $[0,1]^k$ for each $f\in L^2([0,1]^k,\mathbb C)$ (see Theorem~6.11 of \cite{weidmann1980}). Furthermore,
\begin{equation}\label{eq:hsn}
	|\mkern-1.5mu|\mkern-1.5mu|A|\mkern-1.5mu|\mkern-1.5mu|_2^2=\|k_A\|_2^2=\int_{[0,1]^k}\int_{[0,1]^k}|k_A(x,y)|^2dxdy,
\end{equation}
\begin{equation}\label{eq:hsip}
	\langle A,B\rangle_{\mathrm{HS}}
	=\langle k_A,k_B\rangle
	=\int_{[0,1]^k}\int_{[0,1]^k}k_A(x,y)\overline{k_B(x,y)}dxdy
\end{equation}
for two Hilbert-Schmidt operators $A:L^2([0,1]^k,\mathbb C)\to L^2([0,1]^k,\mathbb C)$ and $B:L^2([0,1]^k,\mathbb C)\to L^2([0,1]^k,\mathbb C)$ with the kernels $k_A$ and $k_B$ respectively. The adjoint operator $A^*$ is induced by the kernel $k_A^*(x,y)=k_A(y,x)$. A kernel $k_A:[0,1]^{2k}\to\mathbb C$ is called a Hilbert-Schmidt kernel if $k_A\in L^2([0,1]^{2k},\mathbb C)$.

We briefly review the definitions of the spectral density kernel and the spectral density operator that were introduced by \cite{Panaretos_Tavakoli_2013}. The spectral density kernel $f_\omega$ at frequency $\omega\in\mathbb R$ is defined as
\begin{equation}\label{eq:spectdensity}
	f_\omega
    =\frac1{2\pi}\sum_{t \in \mathbb{Z}} \exp(-i\omega t)r_t,
\end{equation}
where the series converges in $L^2([0,1]^2,\mathbb C)$ provided that
\[
	\sum_{t \in \Z}\|r_t\|_2=\sum_{t\in\mathbb Z}\biggl\{\, \int_0^1\int_0^1|r_t(\tau,\sigma)|^2d\tau d\sigma\biggr\}^{1/2}<\infty.
\]
The spectral density kernel is uniformly bounded and uniformly continuous in $\omega$ with respect to $\|\cdot\|_2$ (see Proposition~2.1 of \cite{Panaretos_Tavakoli_2013}. The corresponding spectral density operator $\mathcal{F}_{\omega}:L^2([0,1],\R) \mapsto L^2([0,1],\mathbb{C})$, induced by the spectral density kernel through right integration, is self-adjoint and non-negative definite for all $\omega \in \R$.

\section{White noise testing}\label{sec3} 
\def\theequation{3.\arabic{equation}}
\setcounter{equation}{0}

We want to test if the time series is white noise, i.e., the spectral density operator does not depend on the frequency $\omega\in\mathbb R$. Formally, we write
\begin{equation}\label{h0}
H_0: \mathcal{F}_{\omega} \equiv \mathcal{F}~~a.e.  \hspace{0.2 in} \text{vs.} \hspace{0.2 in} H_a: \mathcal{F}_{\omega} \neq \mathcal{F}
~~~\text{ on a set of positive Lebesgue measure}
\end{equation}
for some operator $\mathcal{F}:L^2([0,1],\R) \to L^2([0,1],\mathbb{C})$. 
Following  \cite{Dette_Kinsvater_Vetter_2011} we propose to measure deviations from white noise by an  $L^2$-distance and consider the problem of approximating $\mathcal{F}_{\omega}$ by a constant self-adjoint Hilbert-Schmidt operator $\mathcal{F}$ (corresponding to a white noise functional process) by the distance function
\begin{equation}\label{m2}
M^2(\mathcal{F}) = \int_{-\pi}^{\pi}\VERT \mathcal{F}_{\omega} - \mathcal{F} \VERT_2^2d\omega.
\end{equation}
Let us define the kernel $\tilde f:[0,1]^2\to\mathbb R$ by setting
\begin{equation}\label{eq:kerneltilde}
	\tilde f(\tau, \sigma)=\frac1{2\pi}\int_{-\pi}^\pi f_\omega(\tau,\sigma)d\omega
\end{equation}
for each $\tau,\sigma\in[0,1]$. 

\begin{remark}
The kernel $\tilde f$ is a symmetric, positive definite Hilbert-Schmidt kernel (i.e.\ $\|\tilde f\|_2<\infty$). In fact it has a simple time domain representation given by $\tilde f=(2\pi)^{-1}r_0$, where $r_0$ is the autocovariance kernel at lag $0$ defined by~\eqref{eq:autocovker}.
\end{remark}

In the next theorem we derive an explicit expression for the distance $M^2(\mathcal F)$, which  shows that the minimum of $M^2(\mathcal F)$ in the class of all Hilbert-Schmidt operators $\mathcal F:L^2([0,1],\mathbb R)\to L^2([0,1],\mathbb C)$ is attained at the operator $\widetilde{\mathcal F}$ defined by
\begin{equation}\label{oper}
\widetilde{\mathcal F} h (\tau) = \int_0^1\tilde f(\tau, \sigma) h(\sigma)d\sigma.
\end{equation}

\begin{thm}\label{prop:M2F}
Suppose that $\mathcal F:L^2([0,1],\mathbb R)\to L^2([0,1],\mathbb C)$ is a Hilbert-Schmidt operator. Then
\[
	M^2(\mathcal F)
    =\int_{-\pi}^\pi|\mkern-1.5mu|\mkern-1.5mu|\mathcal F_{\omega} - \widetilde{\mathcal F}|\mkern-1.5mu|\mkern-1.5mu|_2^2d\omega
    +\int_{-\pi}^\pi|\mkern-1.5mu|\mkern-1.5mu|\widetilde{\mathcal F} - \mathcal F|\mkern-1.5mu|\mkern-1.5mu|_2^2d\omega,
\]
where $M^2$ is the distance function defined by \eqref{m2} and $\widetilde{\mathcal F}$ is the operator defined by \eqref{oper}.
In particular, $M^2(\mathcal F)$ is minimized at $\widetilde{\mathcal F}$.
\end{thm}
\begin{proof}
The fact that the Hilbert-Schmidt operator norm is induced by the Hilbert-Schmidt inner product yields
\[
	|\mkern-1.5mu|\mkern-1.5mu|\mathcal F_{\omega} - \mathcal F|\mkern-1.5mu|\mkern-1.5mu|_2^2
	=|\mkern-1.5mu|\mkern-1.5mu|\mathcal F_{\omega}-\widetilde{\mathcal F}|\mkern-1.5mu|\mkern-1.5mu|_2^2
		+\langle\mathcal F_{\omega}-\widetilde{\mathcal F},\widetilde{\mathcal F}-\mathcal F\rangle_\mathrm{HS}
        +\langle\widetilde{\mathcal F}-\mathcal F,\mathcal F_{\omega}-\widetilde{\mathcal F}\rangle_\mathrm{HS}
		+|\mkern-1.5mu|\mkern-1.5mu|\widetilde{\mathcal F} - \mathcal F|\mkern-1.5mu|\mkern-1.5mu|_2^2.
\]
Using expression~\eqref{eq:hsip} for the Hilbert-Schmidt inner product and changing the order of integration, we obtain
\[
	\int_{-\pi}^\pi\langle\mathcal F_{\omega}-\widetilde{\mathcal F},\widetilde{\mathcal F}-\mathcal F\rangle_\mathrm{HS}d\omega
	=\int_0^1\int_0^1\int_{-\pi}^\pi[f_\omega(x,y)-\tilde{f}(x,y)]d\omega[\tilde{f}(y,x)-\overbar{f(x,y))}]dxdy
    =0.
\]
The interchange of the order of integration is justified by noticing that
\[
	\int_{-\pi}^\pi\int_0^1\int_0^1\vert[f_\omega(x,y)-\tilde{f}(x,y)][\tilde{f}(y,x)-\overbar{f(x,y)}]\vert dxdyd\omega
	\le2\pi|\mkern-1.5mu|\mkern-1.5mu|\widetilde{\mathcal F} - \mathcal F^*|\mkern-1.5mu|\mkern-1.5mu|_2\sup_{\omega\in[-\pi,\pi]}|\mkern-1.5mu|\mkern-1.5mu|\mathcal F_{\omega} - \widetilde{\mathcal F}|\mkern-1.5mu|\mkern-1.5mu|_2.
\]
A similar argument shows that
\[
	\int_{-\pi}^\pi\langle \widetilde{\mathcal F}-\mathcal F,\mathcal F_{\omega}-\widetilde{\mathcal F}\rangle_\mathrm{HS}d\omega=0,
\]
which completes the proof.
\end{proof}
The next lemma gives us an expression of the minimal distance $M^2(\widetilde{\mathcal F})$ in terms of the spectral density kernel $f_\omega$.
\begin{lemma}
The minimal distance $M^2(\widetilde{\mathcal F})$ is given by
\begin{equation}\label{dist}
	M^2(\widetilde{\mathcal F})
    =\int_0^1\int_0^1\int_{-\pi}^\pi|f_\omega(\tau,\sigma)|^2d\omega d\tau d\sigma
	-\frac1{2\pi}\int_0^1\int_0^1\Bigl|\int_{-\pi}^\pi f_\omega(\tau,\sigma)d\omega\Bigr|^2d\tau d\sigma,
\end{equation}
where $f_\omega$ is the spectral density kernel defined by \eqref{eq:spectdensity}.
\end{lemma}
\begin{proof}
Using~\eqref{eq:hsn} and changing the order of integration it follows that
\begin{equation}\label{h1}
M^2(\widetilde{\mathcal F})
	=\int_{-\pi}^\pi|\mkern-1.5mu|\mkern-1.5mu|\mathcal F_{\omega} - \widetilde{\mathcal F}|\mkern-1.5mu|\mkern-1.5mu|_2^2d\omega
	=\int_0^1\int_0^1\int_{-\pi}^\pi|f_\omega(\tau,\sigma)-\tilde{f}(\tau,\sigma)|^2d\omega d\tau d\sigma.
\end{equation}
Since
\[
	|f_\omega(\tau,\sigma)-\tilde{f}(\tau,\sigma)|^2
	=|f_\omega(\tau,\sigma)|^2
	-f_\omega(\tau,\sigma)\overbar{\tilde f(\tau,\sigma)}
	-\tilde f(\tau,\sigma)\overbar{f_\omega(\tau,\sigma)}
	+|\tilde f(\tau,\sigma)|^2,
\]
we obtain
\begin{align*}
	\int_{-\pi}^\pi|f_\omega(\tau,\sigma)-\tilde{f}(\tau,\sigma)|^2d\omega
	&=\int_{-\pi}^\pi|f_\omega(\tau,\sigma)|^2d\omega
	-\int_{-\pi}^\pi f_\omega(\tau,\sigma)d\omega\overbar{\tilde{f}(\tau,\sigma)}\\
	&\qquad-\tilde{f}(\tau,\sigma)\int_{-\pi}^\pi\overbar{f_\omega(\tau,\sigma)}d\omega
	+\int_{-\pi}^\pi|\tilde{f}(\tau,\sigma)|^2d\omega\\
	&=\int_{-\pi}^\pi|f_\omega(\tau,\sigma)|^2d\omega
	-\frac1{2\pi}\Bigl|\int_{-\pi}^\pi f_\omega(\tau,\sigma)d\omega\Bigr|^2,
\end{align*}
and the assertion follows from equation~\eqref{h1}.
\end{proof}
\begin{remark}\label{remalt} ~~
\begin{enumerate}[(a)]
\item The minimum distance  can be expressed in time domain as
\begin{equation}\label{distsalt}
M^2(\widetilde{\mathcal F})={1\over 2\pi }\sum_{t\ne 0}\|r_t\|_2^2 = {1\over \pi }\sum_{t=1}^\infty 
\|r_t\|_2^2.
\end{equation}
Representation \eqref{distsalt} follows from \eqref{dist} using \eqref{eq:spectdensity}, \eqref{eq:kerneltilde} and the fact that the functions $\{e_t:t\in\mathbb Z\}$ defined by $e_t(\omega)=(2\pi)^{-1/2}\exp(-i\omega t)$ for each $\omega\in[-\pi,\pi]$ and $t\in\mathbb Z$ are orthonormal in $L^2[-\pi,\pi]$ (the space of square-integrable complex functions on $[-\pi,\pi]$ with the usual inner-product). Representation \eqref{distsalt} clearly shows that  the minimum distance is equal to $0$ if and only if  the functional time series is uncorrelated.
\item There exist several alternative ways to measure deviations from white noise and we mention exemplarily the scale invariant measure
\begin{equation}\label{distscale}
	M^2(\widetilde{\mathcal F})
	=\int_0^1\int_0^1 {\frac{1}{2\pi} 
    \int_{-\pi}^\pi |f_\omega(\tau,\sigma)|^2d\omega \over
 \bigl|\frac1{2\pi} \int_{-\pi}^\pi f_\omega(\tau,\sigma)d\omega\bigr|^2}
    d\tau d\sigma-1
    =\sum_{t\ne 0}\int_0^1\int_0^1 { |r_t(\tau,\sigma)|^2 \over |r_0(\tau,\sigma)|^2 }d\tau d\sigma. 
\end{equation}
For the sake of brevity we concentrate on the measure \eqref{dist}, but similar results can be derived for the alternative measure  \eqref{distscale} as well.
\end{enumerate}
\end{remark}
 
For the  estimation of the minimal distance $M_0^2:=M^2(\widetilde{\mathcal F})$, we avoid a direct estimation of the spectral density operator and propose to use the sums of periodograms. More precisely, we consider  the functional discrete Fourier transform (fDFT) of the data $\{X_t\}_{t=0}^{T-1}$ defined as
\begin{equation}\label{eq:fdft}
\wtilde{X}_{\o}^{(T)}(\tau):=\frac1{\sqrt{2\pi T}}\sum_{t=0}^{T-1}X_t(\tau)\exp(-i\o t)
\end{equation}
and consider the periodogram kernel
\[
	p_{\omega}^{(T)}(\tau,\sigma)
    :=[\wtilde{X}_{\o}^{(T)}(\tau)][\overline{\wtilde{X}_{\o}^{(T)}(\sigma)}] = \wtilde{X}_{\o}^{(T)}(\tau)\wtilde{X}_{-\o}^{(T)}(\sigma).
\] 
The estimator of $M_0^2$ is then defined by 
\[
	\widehat{M_T^2}
    =2\pi\int_0^1 \int_0^1(S_{T,2}(\tau,\sigma) - S_{T,1}(\tau,\sigma)\overbar{S_{T,1}(\tau,\sigma)})d\tau d\sigma
\]
where
\begin{equation}\label{eq:ST1ST2}
	S_{T,1}
    =\frac1T\sum_{k=1}^{\lfloor T/2\rfloor}(p_{\omega_k}^{(T)}+ \bar p_{\omega_k}^{(T)})
	\quad\text{and}\quad
    S_{T,2}
    =\frac2T\sum_{k=2}^{\lfloor T/2\rfloor} p_{\omega_k}^{(T)}\bar p_{\omega_{k-1}}^{(T)}
\end{equation}
with
\[
	\omega_k
    =\frac{2\pi k}T
\]
for $k=1,\ldots, \lfloor T/2\rfloor$. The particular form of $S_{T,2}$ in \eqref{eq:ST1ST2} with $\omega_k$ and $\omega_{k-1}$ is chosen to make the estimator $\widehat{M_T^2}$ asymptotically unbiased.

The definition of $\widehat {M_T^2}$ is motivated by Proposition~2.6 and Theorem~2.7 of~\cite{Panaretos_Tavakoli_2013} which states that
\[
	\E(p_{\omega_k}^{(T)}(\tau,\sigma)) \approx f_{\omega_k}(\tau,\sigma)\quad\text{and}\quad\operatorname{Cov}(p_{\omega_k}^{(T)}(\tau_1,\sigma_1),p_{\omega_l}^{(T)}(\tau_2,\sigma_2)) \approx 0
\]
in $L^2$ for $k,l \in \{1,2,\dots, \lfloor T/2 \rfloor\}$ and $k \neq l$. Therefore using the fact that $\overline{f_{\omega}} = f_{-\omega}$ we have
$$\E(S_{T,1}(\tau,\sigma)) \approx \frac{1}{T}\sum_{k=1}^{\lfloor T/2 \rfloor}(f_{\omega_k}(\tau,\sigma) + \overline{f_{\omega_k}(\tau,\sigma)}) \approx \frac{1}{2\pi}\int_{-\pi}^{\pi}f_{\omega}(\tau,\sigma)d\omega$$
$$\E(S_{T,2}(\tau,\sigma)) \approx \frac{2}{T}\sum_{k=1}^{\lfloor T/2 \rfloor} f_{\omega_k}(\tau,\sigma)\overline{f_{\omega_{k-1}}(\tau,\sigma)} \approx \frac{2}{2\pi}\int_{0}^{\pi}\vert f_{\omega}(\tau,\sigma)\vert^2d\omega = \frac{1}{2\pi}\int_{-\pi}^{\pi}\vert f_{\omega}(\tau,\sigma)\vert^2 d\omega ~.$$
This heuristically motivates the approximation $\E(\widehat{M_T^2})\approx M_0^2=M^2(\widetilde{\mathcal F}) $ and the use of  $\widehat {M_T^2}$ as an   estimator of the minimal distance $M_0^2$.

\begin{remark}\label{rmk:real}
The estimator $\widehat {M_T^2}$ is real-valued since
\[
	\widehat{M_T^2}
	=2\pi\biggl[\frac2T\sum_{k=2}^{\lfloor T/2\rfloor}\langle p_{\omega_k}^{(T)},p_{\omega_{k-1}}^{(T)}\rangle
	-\|S_{T,1}\|_2^2\biggr]
\]
and $\langle p_{\omega_k}^{(T)},p_{\omega_{k-1}}^{(T)}\rangle=|\langle\widetilde X_{\omega_k}^{(T)},\widetilde X_{-\omega_{k-1}}^{(T)}\rangle|^2$.
\end{remark} 
\medskip

\begin{remark}  ~
\begin{enumerate}[(a)]
\item If we change the lower index of the sum in the definition of  $S_{T,1}$ from $1$ to $-\lfloor(T-1)/2\rfloor$, we obtain
\[
	\frac1T\sum_{k=-\lfloor(T-1)/2\rfloor}^{\lfloor T/2\rfloor}(p_{\omega_k}^{(T)}(\tau,\sigma)+ \bar p_{\omega_k}^{(T)}(\tau,\sigma))=\frac1{\pi T}\sum_{t=0}^{T-1}X_t(\tau)X_t(\sigma)
\]
for $T\ge1$ and $\tau,\sigma\in[0,1]$. However, it is not clear if $S_{T,2}$ has such a nice representation in the time domain.
\item  The representation \eqref{distsalt} suggests an alternative estimate of  the minimum distance $M^2(\widetilde{\mathcal F})$  in the time domain
 , that is 
 \[
	\widetilde {M_T^2}
    ={1\over \pi }\sum_{t=1}^{p_T} \| \widehat r_t\|_2^2 
 \]
 where $p_T \to \infty$ is a sequence of positive integers and 
 $\widehat r_t$ is a an appropriate estimator of the autocovariance kernel defined in
 \eqref{eq:autocovker}, for example,
 $$
 \hat r_t(\tau,\sigma)= 
 {1 \over T-k} \sum_{t=1}^{T-k} \big (X_t (\tau,\sigma) -\bar X_T(\tau,\sigma)\big ) 
  \big(X_{t-k} (\tau,\sigma) -\bar X_T(\tau,\sigma) \big ).
$$
The performance of test based on the estimator $\widetilde {M_T^2}$ will depend sensitively on the 
choice of the  sequence $p_T$, while the approach proposed here does not require the specification
of a regularization parameter.
\end{enumerate}
\end{remark}

The next theorem is our main result and formalizes heuristic arguments.
It shows that $\widehat{M_T^2} $ is a consistent estimator of $M_0^2$
and that
an appropriately standardized version of $\widehat {M_T^2}$ is asymptotically normal distributed under the  null hypothesis and the alternative. The proof  is complicated and therefore deferred to  Section \ref{sec5} and \ref{sec6}.

\begin{thm} \label{thm2}
Suppose that $\{X_k\}_{k\in\mathbb Z}$ is a strictly stationary time series with values in $L^2([0,1],\mathbb R)$, 
$\E\|X_0\|_2^k<\infty$ for each $k\ge1$,
\begin{enumerate}[(i)]
\item\label{cond:tight}the integral
\[
	\int_0^1\int_0^1\sum_{t_1,t_2,t_3=-\infty}^\infty|\E[X_{t_1}(\tau)X_{t_2}(\sigma)X_{t_3}(\tau)X_0(\sigma)]|d\tau d\sigma
\]
is finite,
\item\label{cond:mixing} $\sum_{t_1,t_2,\ldots,t_{k-1} = -\infty}^{\infty}(1 + \vert t_j \vert) \| \cm(X_{t_1},\ldots,X_{t_{k-1}},X_{0})\|_2 < \infty$ for $j= 1,2, \ldots, k-1$ and all $k\ge1$.
\end{enumerate}
Then
\[
	\sqrt T(\widehat {M_T^2}- M_0^2) 
 	\stackrel{d}{\to} N(0, v^2)\quad\text{as}\quad T\to\infty,
\]
where the asymptotic variance $v^2$ is given by 
 \begin{eqnarray} \label{v2}
v^2&=& 16\pi\int\limits_{[0,1]^4}\int_{-\pi}^{\pi}f_{\omega}(\tau_1,\sigma_1)f_{\omega}(\sigma_1,\tau_2)f_{\omega}(\tau_2,\sigma_2)f_{\omega}(\sigma_2,\tau_1)d\omega d\tau_1d\sigma_1d\tau_2d\sigma_2\nonumber\\
&+ &4\pi\int\limits_{[0,1]^4}\int_{-\pi}^{\pi}\vert f_{\omega}(\tau_1,\sigma_1)f_{\omega}(\tau_2,\sigma_2) \vert^2d\omega d\tau_1d\sigma_1d\tau_2d\sigma_2\nonumber\\
&+ &8\pi \int\limits_{[0,1]^4}\int\limits_{[-\pi,\pi]^2}f_{\omega_1}(\tau_1,\sigma_1)f_{\omega_2}(\tau_2,\sigma_2)f_{\omega_1,-\omega_1,\omega_2}(\sigma_1,\tau_1,\sigma_2,\tau_2)d\omega_1d\omega_2\tau_1d\tau_2d\sigma_1d\sigma_2\nonumber\\
&- &16\int\limits_{[0,1]^4}\int\limits_{[-\pi,\pi]^2}f_{\omega_1}(\tau_1,\sigma_1)f_{\omega_2}(\sigma_1,\tau_2)f_{\omega_2}(\tau_2,\sigma_2)f_{\omega_2}(\sigma_2,\tau_1)d\omega_1d\omega_2d\tau_1d\sigma_1d\tau_2d\sigma_2\nonumber\\
&- &4\int\limits_{[0,1]^4}\int\limits_{[-\pi,\pi]^3}f_{\omega_1}(\tau_1,\sigma_1)f_{\omega_2}(\tau_2,\sigma_2)f_{\omega_3,-\omega_3,\omega_2}(\sigma_1,\tau_1,\sigma_2,\tau_2)d\omega_1d\omega_2d\omega_3d\tau_1d\tau_2d\sigma_1d\sigma_2\nonumber\\
& +&\frac{4}{\pi}\int\limits_{[0,1]^4}\int\limits_{[-\pi,\pi]^3}f_{\omega_1}(\tau_1,\sigma_1)f_{\omega_2}(\tau_2,\sigma_2)f_{\omega_3}(\tau_1,\sigma_2)f_{\omega_3}(\tau_2,\sigma_1)d\omega_1d\omega_2d\omega_3d\tau_1d\sigma_1d\tau_2d\sigma_2\nonumber\\
&+ &\frac{2}{\pi}\int\limits_{[0,1]^4}\int\limits_{[-\pi,\pi]^4}f_{\omega_1}(\tau_1,\sigma_1)f_{\omega_2}(\tau_2,\sigma_2)f_{\omega_3,-\omega_3,\omega_4}(\sigma_1,\tau_1,\sigma_2,\tau_2)d\omega_1d\omega_2d\omega_3d\omega_4d\tau_1d\tau_2d\sigma_1d\sigma_2\nonumber\\ \label{v2h1}
\end{eqnarray}
Moreover, under the null hypothesis the asymptotic variance simplifies to 
 \begin{eqnarray} \label{v2h0}
 v^2_{H_0} =  8\pi^2\left(\int_{[0,1]^2}\big\vert f_0(\tau,\sigma)\big\vert^2d\tau d\sigma \right)^2. 
 \end{eqnarray}
 \end{thm}
 
 \medskip
 
\noindent
\begin{remark}
The assumptions that we use are rather strong, but we do not make any structural assumptions. We prove \autoref{thm2} (the proof is in \autoref{sec5} and \autoref{sec6}) by establishing convergence of certain random elements in $L^1([0,1]^2,\mathbb C)$ and using the continuous mapping theorem. Assumption~\eqref{cond:tight} is used to establish tightness of certain random elements in $L^1([0,1]^2,\mathbb C)$ and assumption~\eqref{cond:mixing} is used to establish the convergence of the finite dimensional distributions (fdds) of the same random elements. We establish the convergence of the fdds by showing that a joint cumulant of any order $k$ greater than $2$ goes to $0$ as $T\to\infty$. That is why we need assumption~\eqref{cond:mixing} for all $k\ge1$.
\end{remark}
\begin{remark}
 Note that the hypotheses in \eqref{h0} can be rewritten as $$H_0: M_0^2=0 \hspace{0.3 in} \text{vs.} \hspace{0.3 in} H_a: M_0^2 > 0.$$ Therefore Theorem \ref{thm2} provides a very simple test for these hypotheses by rejecting the null hypothesis $H_0$ for large values of $ \widehat {M_T^2}$. 
 
Therefore we test the hypotheses \eqref{h0} by rejecting the null hypothesis
 of a functional white noise process whenever\begin{eqnarray} \label{testwn}
\widehat{M^2_T} ~> ~  \frac{\widehat{v_{H_0}}}{\sqrt{T}} u_{1-\alpha}~,
 \end{eqnarray}
 where $u_{1-\alpha}$ denotes the $({1-\alpha})$-quantile of the standard normal distribution and $\widehat{v_{H_0}}$ is the square root of an appropriate
 estimator of  the asymptotic variance under the null hypothesis.
 \end{remark}

   Because Theorem \ref{thm2} is also valid under the alternative the test
 \eqref{testwn} is obviously consistent. Moreover, Theorem \ref{thm2} also provides a simple approximation of the power of the test, that is 
   \begin{eqnarray} \label{power}
 \mathbb{P} \Bigl(\widehat{M^2_T} >   \frac{\widehat{v_{H_0}}}{\sqrt{T}}u_{1-\alpha} \,\Bigr) ~  \approx ~ 
 \Phi \Bigl ( \sqrt{T}\ \frac {M_0^2}{\nu_{H_1}} - \frac {v_{H_0}}{v_{H_1}} \ u_{1- \alpha} \Bigr
),
 \end{eqnarray}
where $v_{H_0}$ and $v_{H_1}$ denote the (asymptotic) standard deviation of $\sqrt{T} \widehat{M^2_T}$
 under the null hypothesis
and alternative, respectively, and $\Phi$ is the distribution function of the standard normal distribution.

\begin{remark} ~
\begin{enumerate}[(a)] 
\item Under the null hypothesis the variance does not involve fourth order cumulants. Note that $2\pi\E(S_{T2}(\tau,\sigma)) = \int_{-\pi}^{\pi} |f_{\omega}(\tau,\sigma)|^2 d\omega$, and therefore a consistent estimator of the standard deviation under the null hypothesis is given by \begin{equation}
\label{var_H0}
\widehat{v_{H_0}} = 4\pi\int_0^1\int_0^1 S_{T,2}(\tau,\sigma)d\tau d\sigma.
\end{equation}
\item To obtain $\hat v^2_{H_1}$, an estimate of $v^2$ under the alternative hypothesis in 
Theorem \ref{thm2}, each term in the expression  given in \eqref{v2h1}  is estimated by taking sums over different frequencies 
for appropriate products of periodograms. 
 For example the first term  
 $$
 16\pi\int_{[0,1]^4}\int_{-\pi}^{\pi}f_{\omega}(\tau_1,\sigma_1)f_{\omega}(\sigma_1,\tau_2)f_{\omega}(\tau_2,\sigma_2)f_{\omega}(\sigma_2,\tau_1)d\omega d\tau_1d\sigma_1d\tau_2d\sigma_2$$
of $v^2$ is estimated by
$$\frac{64\pi^2}{T}\int_{[0,1]^4}\sum_{k=4}^{\lfloor N/2 \rfloor} p^{(T)}_{\omega_k}(\tau_1,\sigma_1)p^{(T)}_{\omega_{k-1}}(\sigma_1,\tau_2)p^{(T)}_{\omega_{k-2}}(\tau_2,\sigma_2)p^{(T)}_{\omega_{k-3}}(\sigma_2,\tau_1)d\tau_1d\tau_2d\sigma_1d\sigma_2,$$
and the  other terms are estimated similarly. The terms involving fourth order spectrum are slightly more difficult to estimate
but  can be constructed by considering fourth order periodogram as described in  formula (1.9) of \cite{brillinger1967}.
The details are omitted for the sake of brevity.  Note also that  the terms of  fourth order vanish
 if the functional process has a Gaussian distribution.
\end{enumerate}

\end{remark}

\begin{remark}
Besides the  simple test for the classical hypotheses of the form \eqref{h0}, Theorem  \ref{thm2} has further important statistical 
applications, which will be briefly discussed in this remark.
\begin{enumerate}[(a)]
\item In applications it is  often reasonable 
to work under the white noise assumption  even in cases where the errors show only slight deviations from white noise. In this case a test for the "classical" hypothesis \eqref{h0} is not useful as it rejects the null hypothesis even for small values
of $M_0^2$ if the sample size is sufficiently large. Moreover, if the null hypothesis in \eqref{h0} is not rejected there is no control of the type I error. \\
In order to address these problems we propose to formulate 
 hypotheses in terms of the $L^2$-distance $M_0^2$ and consider {\it precise hypotheses}
as introduced by  \cite{bergdela1987}, that is
\begin{eqnarray} \label{hrel}
H_\Delta :  M_0^2 \leq \Delta  ~~ & \text{ vs. } &  ~~~K_\Delta :  M_0^2 >  \Delta ~, \\
H_\Delta :  M_0^2 \geq \Delta  ~~ & \text{ vs. } &  ~~~K_\Delta :  M_0^2 <  \Delta  ~, \label{hequiv}
\end{eqnarray} 
where $\Delta $ is a pre-specified constant. For $\Delta >0$ we call the alternative in  \eqref{hrel} {\it relevant deviation}
from white noise and note  that the case $\Delta=0$  in \eqref{hrel}
corresponds  to the "classical" hypothesis \eqref{h0}. 
The alternative in \eqref{hequiv} is called {\it similarity} to white noise and of particular importance.
Hypotheses of the type \eqref{hequiv} are useful
if one wants to control the type one error when one works under the assumption of 
a functional white noise error process.  The choice of the threshold $\Delta$ depends on the particular application, but we argue that from a practical point of view it might be very reasonable to think about this choice more carefully and to define the size of deviation in which one is really interested. Precise hypotheses of the form \eqref{hrel} and \eqref{hequiv}  have  nowadays been considered  in various fields of statistical inference including medical, pharmaceutical, chemistry or environmental statistics [see \cite{chowliu1992} or  \cite{mcbride1999} among others].  

In contrast to other methods, the approach motivated by \autoref{thm2} can be easily used to construct a test for hypotheses of this type.
For  the sake of brevity 
we restrict ourselves to the hypothesis  \ref{hequiv} of {\it similarity} to white noise. Then it is easy to see that an 
 asymptotic  level $\alpha$  test for the hypothesis (\ref{hequiv}) is obtained by rejecting
the null hypothesis, whenever
\begin{eqnarray}  \label{prectest}
\widehat{M^2_T} - \Delta  < \frac{ \hat v_{H_1} }{\sqrt{T}} u_{\alpha} \: ,
\end{eqnarray}
where $ \hat v_{H_1}^2$ denotes an estimator of the variance  in \eqref{v2h1} and $u_\alpha$  is 
the $({1-\alpha})$-quantile of the standard normal distribution. 
Note that this procedure allows for accepting the null hypothesis of an  ``approximate'' white noise
at controlled type I error.

\item In a a similar way an application of \autoref{thm2} shows that 
the interval 
$$
\Big[ \max \Big\{ 0, \widehat{M^2_T} - \frac{\hat v_{H_1}}{\sqrt{T}} u_{1-\alpha/2}
\Big\}, \widehat{M^2_T}  + \frac{\hat v_{H_1}}{\sqrt{T}}u_{1-\alpha/2}
\Big]
$$
is an asymptotic confidence interval for the deviation $M^2$ from a white noise
functional process.
\end{enumerate}
\end{remark}

\section{Finite sample properties}\label{sec4} 
\def\theequation{4.\arabic{equation}}
\setcounter{equation}{0}

In this section, we investigate the finite sample performance of 
the  method proposed in this paper by means of a simulation study. We have calculated the rejection probabilities of the test  \eqref{testwn} for the sample sizes $T = 128,$ $256,$ $512$ and $1024$, where  the number of Monte Carlo replications is always $1000$. 
For the sake of a comparison our simulation setup is similar to that of \cite{zhang2016} who proposed an alternative test for a functional white noise process.

Under the null hypothesis, we simulate i.i.d. data from a standard Brownian motion, Brownian bridge and data from the FARCH(1) process defined as,
$$X_t(\tau) = \epsilon_t(\tau)\sqrt{\tau+\int_0^1c_{\psi}\exp\left(\frac{\tau^2 + \sigma^2}{2}\right)X_{t-1}^2(\sigma)d\sigma},\hspace{0.3 in} t=1,2,\dots, ~~\tau \in [0,1],$$
where $\{\epsilon_t\}_{t\in \mathbb{Z}}$ is a sequence of i.i.d.\ standard Brownian motions and $c_{\psi} = 0.3418$. As explained in \cite{zhang2016} 
observations from the FARCH(1) process are  uncorrelated but dependent. Therefore most of the white noise testing methods (especially non spectra-based procedures) have a large type I error under this setup. The data are generated on a grid of $1000$ equispaced points in $[0,1]$. The kernels $S_{T2}$ and $S_{T1}$ are computed at the $1000 \times 1000$ equispaced grid points on $[0,1]^2$.  The integrals of the kernels are estimated by averaging of the  function values on the grid points. The 
asymptotic variance of the test statistic under the null hypothesis is estimated by \eqref{var_H0}. 

The corresponding results are presented in Table \ref{Sim_size}. 
We observe a very good approximation of the nominal level in all cases under considerations (even for the sample size $T=128$). For the sake of comparison we also display in Table \ref{Sim_size} the simulated level of the bootstrap test proposed in \cite{zhang2016} (numbers in brackets). 
This author used a block bootstrap procedure to generate critical values, which requires  the specification of the block length as a regularization parameter. This parameter was chosen by the minimum volatility index method as described in Section 2.2 (page 79) of \cite{zhang2016}.
For this choice we also  observe a very good approximation of the nominal level in all cases under consideration.

\begin{table}[t]
\begin{center}
\caption{\it Empirical rejection probabilities  (in percentage) of the test \eqref{testwn} under the null hypothesis. The numbers in brackets give the corresponding results of the test of \cite{zhang2016}.}
\label{Sim_size}
\bigskip
\begin{tabular}{|c|c|c|c|c|c|c|c|c|c|}
\hline
& \multicolumn{3}{|c|} {Brownian Motion} & \multicolumn{3}{|c|} {Brownian Bridge} & \multicolumn{3}{|c|} {FARCH(1)}\\
\hline
 T & 10\% & 5\% & 1\%  & 10\% & 5\% & 1\%  & 10\% & 5\% & 1\% \\
\hline
\multirow{2}{*}{128} & 9.5 & 4.8 & 1.1 & 10.8 & 5.3&  0.8 & 11.1 & 5.7 & 0.8\\
& (11.0) & (4.2) & (0.8) & (11.0) & (5.4) & (1.1) & (10.7) & (5.9 ) & (0.9)\\
\hline
\multirow{2}{*}{256} & 9.6 & 5.1 & 1.3 & 10.3 & 5.4 & 0.9 & 10.9 & 5.5 & 0.7\\
& (10.0) & (4.2 ) & (0.9 ) & (9.5) & (4.8) & (0.7) & (11.1) & (5.2) & (0.9)\\ 
\hline
\multirow{2}{*}{512} & 10.1 & 5.1 & 0.8 & 9.7 & 5.1 & 1.0 & 10.9 & 5.3 & 0.8\\
& (9.9) & (4.7) & (0.6) & (10.3) & (5.9) & (1.3) & (11.1) & (4.9) & (0.7) \\
\hline
\multirow{2}{*}{1024} & 9.8 & 4.9 & 0.9 & 9.9 & 5.2 & 0.8 & 10.5 & 5.2 & 0.7\\
& (10.0)& (4.9) & (0.8) & (9.9) & (5.1) & (1.1) & (9.8) & (4.8) & (1.2)\\
\hline
\end{tabular}
\end{center}
\end{table}

\begin{table}
\begin{center}
\caption{\it Empirical rejection probabilities  (in percentage) of the test \eqref{testwn}  under the alternative. The model is a {\rm FAR(1)} model with i.i.d. innovations $\epsilon_t$  from a Brwonian motion or Brwonian bridge and two different integral opertors are considered.
The numbers in brackets give  the corresponding results of the test of \cite{zhang2016}.}
\label{Sim_power}
\medskip
\begin{tabular}{|c|c|c|c|c|c|c|}
\hline
$\epsilon_t$ & \multicolumn{6}{|c|} {Brownian motion} \\
\hline
$\mathcal{K} $& \multicolumn{3}{|c|} {\eqref{gausskern}} & \multicolumn{3}{|c|} 
{\eqref{wienkern}}\\
\hline
\hline
 T & 10\% & 5\% & 1\%  & 10\% & 5\% & 1\% \\
\hline
\multirow{2}{*}{128} & 82.6 & 80.7 & 65.9  & 87.6 & 82.4 & 66.9 \\
& (86.1) & (83.7) & (58.5) & (89.9) & (83.1) & (59.7)   \\
\hline
\multirow{2}{*}{256} & 99.0 & 98.2 & 98.2 & 99.4& 98.3 & 94.2 \\
& (99.6) & (99.2) & (99.0) & (99.9) & (99.5) & (98.6) \\
\hline
\multirow{2}{*}{512} & 99.8 & 99.6 & 99.6  & 99.9 & 99.9 & 99.6 \\
& (99.7) & (99.5) & (99.0) & (99.9) & (99.8) & (99.1) \\
\hline
\multirow{2}{*}{1024} & 100.0 & 99.9 & 99.7 & 100.0 & 100.0 & 99.8 \\
& (100.0 ) & (100.0) & (99.8 )& (100.0) & (99.8) & (99.5) \\
\hline
\hline
$\epsilon_t$ & \multicolumn{6}{|c|} {Brownian bridge} \\
\hline
$\mathcal{K} $& \multicolumn{3}{|c|} {\eqref{gausskern}} & \multicolumn{3}{|c|} 
{\eqref{wienkern}}\\
\hline
\hline
 T & 10\% & 5\% & 1\%  & 10\% & 5\% & 1\%  \\
\hline
\multirow{2}{*}{128}  & 80.1 & 77.4 & 60.1 &  87.6 & 79.9 & 61.2\\
&  (79.2) & (68.3) & (54.4) & (80.2) & (65.8) & (58.1)  \\
\hline
\multirow{2}{*}{256} &  100.0 & 97.0 & 95.5 & 99.9 & 98.3 & 98.1\\
& (100.0) & (98.2) & (97.2) & (100.0) & (99.1) & (98.8)\\
\hline
\multirow{2}{*}{512}  & 100.0 & 99.3 & 99.3 &  100.0 & 100.0  & 98.8\\
 & (100.0 ) & (98.7) & (98.1) & (100.0) & (100.0) & (99.1)\\
\hline
\multirow{2}{*}{1024} & 100.0 & 100.0 & 100.0 & 100.0 & 100.0 & 100.0\\
 & (100.0) & (100.0) & (99.4) & (100.0) & (100.0) & (100.0) \\
\hline
\end{tabular}
\end{center}
\end{table}

Under the alternatives, we simulate data from the  FAR(1) model 
\begin{equation}
\label{FAR}
X_t(\tau) - \mu(\tau) = \rho(X_{t-1}-\mu)(\tau) + \epsilon_t(\tau), \hspace{0.1 in} t=1,2,\dots
\end{equation}
where $ \rho$ denotes an integral operator acting on $L^2[0,1]$ defined by 
\begin{equation}
\rho(x)(\tau) = \int_0^1\mathcal{K}(\tau,\sigma)x(\sigma)d\sigma
~,~~x \in L^2[0,1], \label{intop}
\end{equation}
for some kernel $\mathcal{K} \in L^2([0,1]^2)$, and $\{\epsilon_t(\tau)\}_{t\in \mathbb{Z}}$ is a sequence of i.i.d. mean zero innovations in $L^2[0,1]$. For our simulations we use four different FAR(1) models where the innovations are either Brownian motions or Brownian bridges and the kernel 
in the integral operator \eqref{intop} is either the Gaussian kernel  
\begin{align} \label{gausskern}
\mathcal{K}_g(\tau,\sigma) = c_g\exp\left((\tau^2+\sigma^2)/2\right)
\end{align}
or the Wiener kernel 
\begin{align} \label{wienkern}
\mathcal{K}_{w}(\tau,\sigma) = c_w\min(\tau,\sigma),
\end{align}
where the constants  $c_g$ and $c_w$ were chosen such that the corresponding Hilbert-Schmidt norm is equal $0.3$. The corresponding results  of the new test are presented in Table \ref{Sim_power}. We observe very good rejection probabilities in all considered models. A comparison  with the procedure of \cite{zhang2016}
shows that the power of both tests is very similar. Only for small sample sizes we observe small differences between both procedures. While the test  of  \cite{zhang2016} has slightly larger power 
for a FAR(1) model with  a Gaussian kernel and  i.i.d innovations
(in most  cases), 
the new test proposed in this paper usually yields better results for 
a FAR(1) model with  a Wiener  kernel and  i.i.d innovations. 

Our numerical study can be summarized as follows. The new test proposed in this paper   exhibits similar properties as the block bootstrap test suggested in \cite{zhang2016}. The latter approach uses resampling, which is computationally expensive for functional data. Moreover, it 
requires  the specification of the length of the blocks for the bootstrap, and the results may depend on this regularization. In contrast the 
new test does  not need a regularization parameter and  critical values can be directly obtained from the table of the standard normal distribution.
Moreover, the method can be easily extended to test precise hypotheses of the form \eqref{hrel}
or  \eqref{hequiv} and our results can be used to provide confidence intervals for a measure of deviation from white noise.

\section{Conclusions} 
In this paper a  new and simple test  is proposed for the hypothesis that a  functional time series is a  white noise process. The test is based on an
estimate of the  minimal  $L^2$-distance between the spectral density operator and its best  ($L^2$-)approximation by a spectral density operator 
corresponding to a white noise process. The minimal distance can be interpreted as a measure for the deviation from a white noise process (which vanishes
in case of white noise) and estimated by  sums of periodogram kernels.  

Asymptotic normality of a standardized version of this estimator is established under the null hypothesis of a white noise process and under the alternative, such
that  the quantiles of the normal distribution can be used for testing and  power analysis. In particular the test does not require the estimation of a long run variance.
The results are also applicable for the construction of confidence intervals for the measure of deviation from a white noise process and for the construction of tests for precise hypotheses such as the hypothesis of a relevant deviation from white noise.
 
An important problem for future research is to investigate if the methodology can be extended such it can be applied to the residuals from a parametric functional time series model in order to construct goodness-of-fit tests. The method of directly estimating a minimal distance between an object and its best approximation under the null hypothesis is also applicable for other 
testing problems such as testing the independence of two functional time series. A very challenging question in this context is to investigate if our  approach can be used to develop a test for the stationarity of a functional time series and research in this direction is currently underway.

\bigskip
\bigskip

\noindent 	
{\bf Acknowledgments.}
This work has been supported in part by the Collaborative Research Center ``Statistical modelling of nonlinear
dynamic processes'' (SFB 823, Teilprojekt  C1, C3) of the German Research Foundation (DFG). The authors would like to thank two anonymous referees and the Associate Editor for their constructive comments on an earlier version of this paper.

 \bibliography{References}

\newpage
\section{Proof of \autoref{thm2}} \label{sec5} 
\def\theequation{5.\arabic{equation}}
\setcounter{equation}{0}
\medskip
\noindent
We prove that the random elements $\{I_T\}=\{I_T\}_{T\ge1}$ with values in  $L^1([0,1]^2,\mathbb C)$ defined by
\[
	I_T(\tau,\sigma)=\sqrt T\biggl[2\pi[S_{T,2}(\tau,\sigma)-S_{T,1}(\tau,\sigma)\overline {S_{T,1}(\tau,\sigma)}]
    -\int_{-\pi}^{\pi}\vert f_{\omega}(\tau,\sigma)\vert^2d\omega
    +\frac{1}{2\pi}\biggl|\int_{-\pi}^{\pi}f_{\omega}(\tau,\sigma)d\omega\biggr|^2\biggr]
\]
for each $(\tau,\sigma)\in[0,1]^2$ converge in distribution to a zero mean Gaussian random element $\mathcal G$  with values in $L^1([0,1]^2,\mathbb R)$ and the covariance kernel $\nu^2$ given by formula \eqref{cov_ker_IT} in Proposition \ref{ITcov} of the following section, where $S_{T_1}$ and $S_{T,2}$ are defined by~\eqref{eq:ST1ST2}. The proof is based on Theorem~2 of \cite{cremers1986}, which states that $I_T$ converges in distribution to $\mathcal G$ as $T\to\infty$ provided that the following three conditions are fulfilled:

\begin{enumerate}[(I)]
\item the finite dimensional distributions (fdds) of $I_T$ converge to the fdds of $\mathcal G$ a.e.\ as $T\to\infty$;
\item\label{cond:dom} there exists an integrable function $f:[0,1]^2\to[0,\infty)$ such that 
\[
	\E|I_T(\tau,\sigma)|\le f(\tau,\sigma)
\]
for each $T\ge1$ and $(\tau,\sigma)\in[0,1]^2$;
\item for each $(\tau,\sigma)\in[0,1]^2$,
\[
	\E|I_T(\tau,\sigma)|\to\E|\mathcal G(\tau,\sigma)|\quad\text{as}\quad T\to\infty.
\]
\end{enumerate}
We establish sufficient conditions for the convergence of the fdds in \autoref{subsec:fdds}. A sufficient condition for the existence of a non-negative integrable function that satisfies condition \eqref{cond:dom} is established in \autoref{subsec:dom}. Finally, the required convergence of moments is established using the fact that if $I_T(\tau,\sigma)$ converges in distribution to $\mathcal G(\tau,\sigma)$ as $T\to\infty$ and $\sup_{T\ge1}\E|I_T(\tau,\sigma)|^2<\infty$, then $\E|I_T(\tau,\sigma)|<\infty$ and $\E|I_T(\tau,\sigma)|\to\E|\mathcal G(\tau,\sigma)|$ as $T\to\infty$ for each $(\tau,\sigma)\in[0,1]^2$ (see Theorem 25.12 and its Corollary on p.\ 338 of \cite{billingsley1995}). We have that $\sup_{T\ge1}\E|I_T(\tau,\sigma)|^2<\infty$ since $\E|I_T(\tau,\sigma)|^2 = \cm_2(I_T(\tau,\sigma))= O(1)$ as $T\to\infty$ (by (I)).

We write
$$\sqrt{T}(\widehat{M^2} - M_0^2) = 2\pi\int_0^1\int_0^1I_T (\tau,\sigma)d\tau d\sigma.$$

As  the map $I:L^2([0,1]^2,\mathbb{C}) \to \mathbb{C}$ defined by  $I(f) = \int_0^1\int_0^1f(\tau,\sigma)d\tau d\sigma$ is continuous,
an application of  the  continuous mapping  Theorem gives
$$\sqrt{T}(\widehat{M^2} - M_0^2) \stackrel{d}{\to} 2\pi\int_0^1 \int_0^1 \mathcal{G}(\tau,\sigma)d\tau d\sigma.$$

Therefore we finally have   
$$\sqrt{T}(\widehat{M^2} - M_0^2) \stackrel{d}{\to}N\left(0,4\pi^2\int_{0}^1\int_{0}^1\int_{0}^1\int_{0}^1\nu^2((\tau_1,\sigma_1)(\tau_2,\sigma_2))d\tau_1d\sigma_1d\tau_2d\sigma_2\right),$$
where the kernel $\nu^2$ is defined in Proposition \ref{ITcov} of the following section.
The asymptotic variance in Theorem \ref{thm2} is now obtained by a straightforward calculation of the integral observing the representation \eqref{cov_ker_IT}.
Under $H_0$ the spectral densities $f_{\omega}$ and  $f_{\omega_1,-\omega_1,\omega_2}$ are free of $\omega$ and $\omega_1, \omega_2$ respectively. 
Therefore under $H_0$ the limiting variance simplifies to:
\begin{align}
\nu^2_{H_0} = &8\pi^2 \int_0^1\int_0^1\int_0^1\int_0^1\vert f_0(\tau_1,\sigma_2)f_0(\tau_2,\sigma_1)\vert^2 d\tau_1 d\tau_2 d\sigma_1 d\sigma_2\nonumber\\
= &  8\pi^2\left(\int_{[0,1]^2}\big\vert f_0(\tau,\sigma)\big\vert^2d\tau d\sigma \right)^2.\label{eq:var_H0}
\end{align}

\subsection{The convergence of the finite-dimensional distributions}\label{subsec:fdds}
To establish the convergence of the fdds, we need to show that
\[
	\left(\begin{array}{ccc}I_T(\tau_1,\sigma_1), & \ldots & ,I_T(\tau_d,\sigma_d)\end{array}\right)^{\mathrm T}\xrightarrow{d}\left(\begin{array}{ccc}\mathcal G(\tau_1,\sigma_1) ,& \ldots & , \mathcal G(\tau_d,\sigma_d)\end{array}\right)^{\mathrm T}
\]
as $T\to\infty$ for each $d\ge1$ and $(\tau_1,\sigma_1),\ldots,(\tau_d,\sigma_d)\in S$ where $S \subset [0,1]^{2d}$ and $S$ has Lebesgue measure $1$. In order to do that we restrict our attention to the vector 
\begin{align} \label{itilde} 
\wtilde{I}_T(\tau_1,\dots,\tau_d,\sigma_1,\dots,\sigma_d) := 
\sqrt{T}\left(\begin{array}{c}
S_{T,1}(\tau_1,{\sigma}_1) - \E\left(S_{T,1}({\tau}_1,{\sigma}_1)\right)\\
\vdots\\
S_{T,1}(\tau_d,{\sigma}_d) - \E\left(S_{T,1}({\tau}_d,{\sigma}_d)\right)\\
S_{T,2}(\tau_1,{\sigma}_1) - \E\left(S_{T,2}({\tau}_1,{\sigma}_1)\right)\\
\vdots\\
S_{T,2}(\tau_d,{\sigma}_d) - \E\left(S_{T,2}({\tau}_d,{\sigma}_d)\right)
\end{array}\right).
\end{align}
First we show that the aforementioned vector converges in distribution to some $N(0,\Sigma)$ random vector and then use the delta method to obtain the desired result. Here we only deal with the case $d=1$, as the general case can be established similarly with an additional amount of notation.
In order to show that the limit distribution of $\tilde{I}_T$ converges to multivariate normal, we use the Cram\' er-Wold device and show for any vector $c \in \R^{2}$, the random variable $c' \wtilde{I}_T({\tau},{\sigma})$ converges in distribution to $N(0,c'\Sigma c)$ variable. 
\\
For this purpose we prove  that the cumulants of $c'\wtilde{I}_T(\tau,\sigma)$ converge to the cumulants of a normal distribution. The first cumulant is trivially zero. Using the fact that cumulants of order $l$ are invariant under centering for $l \geq 2$, Proposition \ref{prop:cum} shows the convergence of higher order cumulants of $c'\wtilde{I}_T(\tau,\sigma)$ to the cumulants of a normal distribution.

\begin{prop}
\label{prop:cum}
Under assumption \eqref{cond:mixing} of \autoref{thm2} we have,
\begin{align*}
\cm_l\left(c_1\sqrt{T}S_{T,2}(\tau,\sigma)+c_2\sqrt{T}S_{T,1}(\tau,\sigma)\right) & = O(1) \hspace{0.1 in} \rm{for }\hspace{0.1 in} l=2,\\
&=o(1) \hspace{0.1 in} \rm{for }\hspace{0.1 in}  l>2,
\end{align*}
for any $c_1, c_2 \in \R$ and a.e. $(\tau,\sigma) \in [0,1]^2$.
\end{prop} 

\begin{proof} 
Introduce the notation $\cm_{n_1,n_2}(X,Y) = \cm(\underbrace{X,\dots,X}_{n_1 },\underbrace{Y,\dots,Y}_{n_2})$. Using the linearity of cumulants as in Theorem 2.3.1 from \cite{Brillinger_2001} we write
$$
\cm_l\left(c_1\sqrt{T}S_{T,1}(\tau,\sigma)+c_2\sqrt{T}S_{T,2}(\tau,\sigma)\right)
= \sum_{n=0}^l c_1^{n}c_2^{l-n}\cm_{n,l-n}\left(\sqrt{T}S_{T,1}(\tau,\sigma),\sqrt{T}S_{T,2}(\tau,\sigma)\right).
$$
We will show that $\cm_{n,l-n}\left(\sqrt{T}S_{T,1}(\tau,\sigma),\sqrt{T}S_{T,2}(\tau,\sigma)\right)$ is bounded for $l=2$ and converges to $0$ for $l > 2$ for $n=0, \dots, l$. First we will show it for $n=0$ and then use induction on $n$, i.e., assuming the result is true for $n=t-1$, We will show the order of $\cm_{n,l-n}\left(\sqrt{T}S_{T,1}(\tau,\sigma),\sqrt{T}S_{T,2}(\tau,\sigma)\right)$ remains same for $n= t$.\\ 

Using linearity again, we obtain

\begin{align*}
&\cm_{n,l-n}\left(\sqrt{T}S_{T,1}(\tau,\sigma),\sqrt{T}S_{T,2}(\tau,\sigma)\right)\\
= &\cm_{n,l-n}\left(\frac{1}{\sqrt{T}}\sum_{k=1}^{\lfloor T/2\rfloor}(p_{\omega_k}(\tau,\sigma) + p_{\omega_k}(\sigma,\tau)),\frac{2}{\sqrt{T}} \sum_{k=1}^{\lfloor T/2\rfloor} p_{\omega_k}^{(T)}(\tau,\sigma)p_{\omega_{k-1}}^{(T)}(\sigma,\tau)\right)\\
= &\cm_{n,l-n}\left(\frac{1}{\sqrt{T}}\sum_{k=1}^{\lfloor T/2\rfloor}p_{\omega_k}(\tau,\sigma) ,\frac{2}{\sqrt{T}} \sum_{k=1}^{\lfloor T/2\rfloor} p_{\omega_k}^{(T)}(\tau,\sigma)p_{\omega_{k-1}}^{(T)}(\sigma,\tau)\right)\nonumber\\
&+\cm_{n,l-n}\left(\frac{1}{\sqrt{T}}\sum_{k=1}^{\lfloor T/2\rfloor}p_{\omega_k}(\sigma,\tau) ,\frac{2}{\sqrt{T}} \sum_{k=1}^{\lfloor T/2\rfloor} p_{\omega_k}^{(T)}(\tau,\sigma)p_{\omega_{k-1}}^{(T)}(\sigma,\tau)\right)\nonumber\\
= &C_{n1} + C_{n2}
\end{align*}
We will argue only for the first term, the second term can be handled similarly.
\begin{align*}
C_{n1} = &\cm_{n,l-n}\left(\frac{1}{\sqrt{T}}\sum_{k=1}^{\lfloor T/2\rfloor}p_{\omega_k}(\tau,\sigma) ,\frac{2}{\sqrt{T}} \sum_{k=1}^{\lfloor T/2\rfloor} p_{\omega_k}^{(T)}(\tau,\sigma)p_{\omega_{k-1}}^{(T)}(\sigma,\tau)\right)\nonumber\\
= &\frac{2^{l-n}}{T^{l/2}}\sum_{k_1,\ldots,k_l=1}^{\lfloor T/2\rfloor}\cm\left(p_{\omega_{k_1}}^{(T)}(\tau,\sigma),\dots,p_{\omega_{k_n}}^{(T)}(\tau,\sigma), p_{\omega_{k_{n+1}}}^{(T)}(\tau,\sigma)p_{\omega_{k_{n+1}-1}}^{(T)}(\sigma,\tau),\dots, p_{\omega_{k_l}}^{(T)}(\tau,\sigma)p_{\omega_{k_l-1}}^{(T)}(\sigma,\tau)\right)\\
= &\frac{2^{l-n}}{T^{l/2}}\sum_{k_1,\ldots,k_l=1}^{\lfloor T/2\rfloor} \cm(Z_{k_1,1}Z_{k_1,2}, \dots, Z_{k_n,1}Z_{k_n,2}, Z_{k_{n+1},1}Z_{k_{n+1},2}Z_{k_{n+1},3}Z_{k_{n+1},4}\dots Z_{k_l,1}Z_{k_l,2}Z_{k_l,3}Z_{k_l,4})
\end{align*}
where
$Z_{i1}:= \wtilde{X}_{\omega_i}^{(T)}(\tau), \ 
Z_{i2}:=  \wtilde{X}_{-\omega_i}^{(T)}(\sigma), \ 
Z_{i3}:=  \wtilde{X}_{\omega_{i-1}}^{(T)}(\sigma) \ \text{and} \
Z_{i4}:=  \wtilde{X}_{-\omega_{i-1}}^{(T)}(\tau)$.
Now using Theorem 2.3.2 from \cite{Brillinger_2001} we write

\begin{equation}
\label{cum_partt}
C_{n1}
= \frac{2^{l-n}}{T^{l/2}}\sum_{k_1=1}^{\lfloor T/2\rfloor} \dots\sum_{k_l=1}^{\lfloor T/2\rfloor}\sum_{\nu}\cm\left(Z_{ij} : ij \in \nu_1\right) \dots \cm\left(Z_{ij} : ij \in \nu_p\right) = \sum_{\nu} C_n(\nu)
\end{equation}
with
$$C_{n}(\nu) = \frac{2^{l-n}}{T^{l/2}}\sum_{k_1=1}^{\lfloor T/2\rfloor} \dots\sum_{k_l=1}^{\lfloor T/2\rfloor}\cm\left(Z_{ij} : ij \in \nu_1\right) \dots \cm\left(Z_{ij} : ij \in \nu_p\right)$$

for all indecomposable partitions $\nu = \nu_1 \cup \nu_2 \cup \dots \cup \nu_p$ of the table
\begin{equation}
\label{table_indext} T_n = 
\left\{
\begin{array}{cccc}
(k_1,1) & (k_1,2) &  &  \\
\vdots &\vdots & &\\
(k_n,1) & (k_n,2) & &\\
(k_{n+1},1) & (k_{n+1},2) & (k_{n+1},3) & (k_{n+1},4) \\
\vdots &\vdots &\vdots &\vdots \\
(k_l,1) & (k_l,2) & (k_l,3) & (k_l,4) 
\end{array}
\right.
\end{equation}

As there are finitely many partitions, we will show that $C_n(\nu)$ is of right order for all indecomposable partition of $T_n$ and $n=0, 1, \dots, l$.
To this end, we claim that $C_0(\nu) = O(1)$ for $l=2$ and $o(1)$ for $l >2$ for all indecomposable partitions $\nu$ of $T_0$. The proof of this claim is presented in Section \ref{subsec:prop_cm}.

Now suppose that we have proved $C_n(\nu)$ has the right order ($O(1)$ for $l=2$ and $o(1)$ for $l>2$) for $n=0, \dots, t-1$. Let $\nu = \nu_1 \cup \dots \cup \nu_p$ be an indecomposable partition of table $T_t$ and $\nu_{p+1} = \{(k_t,3),(k_t,4)\}$. Then $\nu' = \nu_1 \cup \dots \cup \nu_p \cup \nu_{p+1}$ is an indecomposable partition of $T_{t-1}$. Using Theorem B.2 from \cite{Panaretos_Tavakoli_2013_Supp} we get $\cm(Z_{ij}:ij \in \nu_{p+1}) = O(1)$ and
\begin{align*}
C_{t-1}(\nu') = &\frac{2^{l-t+1}}{T^{l/2}}\sum_{k_1=1}^{\lfloor T/2\rfloor} \dots\sum_{k_l=1}^{\lfloor T/2\rfloor}\cm\left(Z_{ij} : ij \in \nu_1\right) \dots \cm\left(Z_{ij} : ij \in \nu_p\right)O(1)\\ 
=&2C_{t}(\nu) \times O(1).
\end{align*}
Therefore $C_t(\nu)$ is of the right order and hence the result is true.
\end{proof}

Note that $(I_T(\tau_1,\sigma_1),\dots,I_T(\tau_d,\sigma_d)) = g(\tilde{I}_t(\tau_1,\dots,\sigma_d))$ where $g:\R^{2d} \to \R^d$, defined as $g(x_1,x_2,\dots,x_{2d}) = (x_{d+1} - x_1^2, \dots, x_{2d}-x_d^2)$. Therefore an application of the delta method (Theorem 8.22 from \cite{lehmann_casella}) along with Lemma \ref{Smean} establishes the convergence of the fdds of $I_T$ to the fdds of $\mathcal G$ almost everywhere as $T\to\infty$.

\subsection{The dominating function}\label{subsec:dom}
We establish a sufficient condition for the existence of a non-negative integrable function that satisfies condition \eqref{cond:dom}.
\begin{thm}\label{thm:dom}
There exists an integrable function $f:[0,1]^2\to[0,\infty)$ such that
\[
	\E|I_T(\tau,\sigma)|\le f(\tau,\sigma)
\]
for each $T\ge1$ and $(\tau,\sigma)\in[0,1]^2$ if
\[
	\int_0^1\int_0^1\sum_{t_1,t_2,t_3=-\infty}^\infty|\E[X_{t_1}(\tau)X_{t_2}(\sigma)X_{t_3}(\tau)X_0(\sigma)]|d\tau d\sigma<\infty.
\]
\end{thm}
We need an auxiliary lemma to prove \autoref{thm:dom}.
\begin{lemma}\label{lemma:boundforproduct}
For each $(\tau,\sigma)\in[0,1]^2$, $\omega,\lambda\in\mathbb R$ and $T\ge1$,
\[
	\E|p_{\omega}^{(T)}(\tau,\sigma)p_{\lambda}^{(T)}(\tau,\sigma)|
	\le\frac7{(2\pi)^2T}\sum_{t_1,t_2,t_3=-\infty}^\infty|\E[X_{t_1}(\tau)X_{t_2}(\sigma)X_{t_3}(\tau)X_0(\sigma)]|.
\]
\end{lemma}
\begin{proof}
By the Cauchy-Schwarz inequality,
\[
	\E|p_{\omega}^{(T)}(\tau,\sigma)p_{\lambda}^{(T)}(\tau,\sigma)|\le(\E|p_\omega^{(T)}(\tau,\sigma)|^2)^{1/2}(\E|p_\lambda^{(T)}(\tau,\sigma)|^2)^{1/2}.
\]
We have that
\[
	\E|p_\omega^{(T)}(\tau,\sigma)|^2
	=\E[\widetilde X_\omega(\tau)\widetilde X_{-\omega}(\sigma)\widetilde X_{-\omega}(\tau)\widetilde X_\omega(\sigma)].
\]
The definition of the fDFT, the linearity of the expectation and the stationarity of the sequence $\{X_t\}_{t\in\mathbb Z}$ yield
\begin{align*}
	&\ E[\widetilde X_\omega(\tau)\widetilde X_{-\omega}(\sigma)\widetilde X_{-\omega}(\tau)\widetilde X_\omega(\sigma)]\\
	&=\frac1{(2\pi T)^2}\sum_{u_1,u_2,u_3,u_4=0}^{T-1}\exp(-i\omega(u_1-u_2-u_3+u_4))\E[X_{u_1}(\tau)X_{u_2}(\sigma)X_{u_3}(\tau)X_{u_4}(\sigma)]\\
	&=\frac1{(2\pi T)^2}\sum_{u_1,u_2,u_3,u_4=0}^{T-1}\exp(-i\omega(u_1-u_2-u_3+u_4))\E[X_{u_1-u_4}(\tau)X_{u_2-u_4}(\sigma)X_{u_3-u_4}(\tau)X_0(\sigma)].
\end{align*}
Let $h^T(t)=1$ for $0\le t\le T-1$ and $0$ otherwise. By setting $t_i=u_i-u_4$ for $1\le i\le 3$ and $t=t_4$, we obtain
\begin{align*}
	&\E[\widetilde X_\omega(\tau)\widetilde X_{-\omega}(\sigma)\widetilde X_{-\omega}(\tau)\widetilde X_\omega(\sigma)]\\
	&=\frac1{(2\pi T)^2}\sum_{t_1,t_2,t_3=-(T-1)}^{T-1}\exp(-i\omega(t_1-t_2-t_3))\E[X_{t_1}(\tau)X_{t_2}(\sigma)X_{t_3}(\tau)X_0(\sigma)]\\
	&\quad\times\sum_{t\in\mathbb Z}h^{(T)}(t_1+t)h^{(T)}(t_2+t)h^{(T)}(t_3+t)h^{(T)}(t)\\
	&=\frac1{(2\pi)^2T}\sum_{t_1,t_2,t_3=-(T-1)}^{T-1}\exp(-i\omega(t_1-t_2-t_3))\E[X_{t_1}(\tau)X_{t_2}(\sigma)X_{t_3}(\tau)X_0(\sigma)]\\
	&\quad+\frac1{(2\pi T)^2}\sum_{t_1,t_2,t_3=-(T-1)}^{T-1}\exp(-i\omega(t_1-t_2-t_3))\E[X_{t_1}(\tau)X_{t_2}(\sigma)X_{t_3}(\tau)X_0(\sigma)]\\
	&\quad\times\sum_{t\in\mathbb Z}[h^{(T)}(u_1+t)h^{(T)}(u_2+t)h^{(T)}(u_3+t)h^{(T)}(t)-[h^{(T)}(t)]^4]
\end{align*}
since
\[
	\sum_{t\in\mathbb Z}[h^{(T)}(t)]^4=T.
\]
Using Lemma~F.7 from \citet{Panaretos_Tavakoli_2013_Supp} it follows that 
\begin{align*}
	&\E[\widetilde X_\omega(\tau)\widetilde X_{-\omega}(\sigma)\widetilde X_{-\omega}(\tau)\widetilde X_\omega(\sigma)]\\
	&\le\frac1{(2\pi)^2T}\sum_{t_1,t_2,t_3=-(T-1)}^{T-1}|\E[X_{t_1}(\tau)X_{t_2}(\sigma)X_{t_3}(\tau)X_0(\sigma)]|\\
	&\quad+\frac2{(2\pi T)^2}\sum_{t_1,t_2,t_3=-(T-1)}^{T-1}(|t_1|+|t_2|+|t_3|)|\E[X_{t_1}(\tau)X_{t_2}(\sigma)X_{t_3}(\tau)X_0(\sigma)]|\\
	&\le\frac7{(2\pi)^2T}\sum_{t_1,t_2,t_3=-(T-1)}^{T-1}|\E[X_{t_1}(\tau)X_{t_2}(\sigma)X_{t_3}(\tau)X_0(\sigma)]|\\
	&\le\frac7{(2\pi)^2T}\sum_{t_1,t_2,t_3=-\infty}^\infty|\E[X_{t_1}(\tau)X_{t_2}(\sigma)X_{t_3}(\tau)X_0(\sigma)]|.
\end{align*}
Analogously, we obtain 
\[
	\E[\widetilde X_\lambda(\tau)\widetilde X_{-\lambda}(\sigma)\widetilde X_{-\lambda}(\tau)\widetilde X_\lambda(\sigma)]
	\le\frac7{(2\pi)^2T}\sum_{t_1,t_2,t_3=-\infty}^\infty|\E[X_{t_1}(\tau)X_{t_2}(\sigma)X_{t_3}(\tau)X_0(\sigma)]|.
\]
which completes the proof of Lemma \ref{lemma:boundforproduct}.
\end{proof}
\begin{proof}[Proof of \autoref{thm:dom}]
We have that
\[
	\E|I_T(\tau,\sigma)|\le 2\sqrt T\E|S_{T,2}(\tau,\sigma)|+2\sqrt T\E|S_{T,1}(\tau,\sigma)\overline S_{T,1}(\tau,\sigma)|
\]
using the inequality $\E|\xi-\E\xi|\le2\E|\xi|$ for any random variable $\xi$ such that $\E|\xi|<\infty$. Using the definition of $S_{T,2}$ and the triangle inequality,
\[
	\sqrt T\E|S_{T,2}(\tau,\sigma)|
	\le\frac2{\sqrt T}\sum_{k=1}^{\lfloor T/2\rfloor}\E|p_{\omega_k}^{(T)}(\tau,\sigma)\overline p_{\omega_{k-1}}^{(T)}(\tau,\sigma)|.
\]
The fact that $\overline p_{\omega_{k-1}}^{(T)}(\tau,\sigma)=p_{-\omega_{k-1}}^{(T)}(\tau,\sigma)$ and the bound of \autoref{lemma:boundforproduct}
now yield
\[
	\sqrt T\E|S_{T,2}(\tau,\sigma)|\\
	\le\frac7{(2\pi)^2}\sum_{t_1,t_2,t_3=-\infty}^\infty|\E[X_{t_1}(\tau)X_{t_2}(\sigma)X_{t_3}(\tau)X_0(\sigma)]|
\]
for each $T\ge1$ and $(\tau,\sigma)\in[0,1]^2$. Similarly, using the triangle inequality,
\begin{multline*}
	\E|S_{T,1}(\tau,\sigma)\overline S_{T,1}(\tau,\sigma)|
	\le\frac1{T^2}\biggl[\sum_{k=1}^{\lfloor T/2\rfloor}\sum_{l=1}^{\lfloor T/2\rfloor}\E|p_{\omega_k}^{(T)}(\tau,\sigma)p_{\omega_l}^{(T)}(\tau,\sigma)|\\
	+\sum_{k=1}^{\lfloor T/2\rfloor}\sum_{l=1}^{\lfloor T/2\rfloor}\E|\overline p_{\omega_k}^{(T)}(\tau,\sigma)\overline p_{\omega_l}^{(T)}(\tau,\sigma)|+2\sum_{k=1}^{\lfloor T/2\rfloor}\sum_{l=1}^{\lfloor T/2\rfloor}\E|\overline p_{\omega_k}^{(T)}(\tau,\sigma)p_{\omega_l}^{(T)}(\tau,\sigma)|\biggr]
\end{multline*}
and
\[
	\sqrt T\E|S_{T,1}(\tau,\sigma)\overline S_{T,1}(\tau,\sigma)|
	\le\frac{7}{(2\pi)^2}\sum_{t_1,t_2,t_3=-\infty}^\infty|\E[X_{t_1}(\tau)X_{t_2}(\sigma)X_{t_3}(\tau)X_0(\sigma)]|
\]
using the bound of \autoref{lemma:boundforproduct}. The proof is complete.
\end{proof}

\section{More technical details} \label{sec6} 
\subsection{Limiting Mean and Variance Calculation}\label{subsec:lim_cov}
\def\theequation{6.\arabic{equation}}
\setcounter{equation}{0}
In this Section we calculate the limiting covariance kernel of the process $I_T$. The main result of this section is stated in Proposition \ref{ITcov}.

\begin{lemma}\label{Smean}
Under the assumption \eqref{cond:mixing} of \autoref{thm2}, we have 
$$\sqrt{T}\left(\E(S_{T,2}(\tau,\sigma)) -\frac{1}{2\pi}\int_{-\pi}^{\pi}\vert f_{\omega}(\tau,\sigma)\vert^2d\omega\right) \to 0$$ $$\sqrt{T}\left(\E(S_{T,1}(\tau,\sigma)) - \frac{1}{2\pi}\int_{-\pi}^{\pi}f_{\omega}(\tau,\sigma)d\omega\right) \to 0$$ as $T \to \infty$ for almost every $(\tau,\sigma) \in [0,1]^2$.
\end{lemma}
\begin{proof}
Using Proposition 2.5 and 2.6 from \cite{Panaretos_Tavakoli_2013} we obtain for almost every $(\tau,\sigma) \in [0,1]^2$,
\begin{align}
\E(S_{T,2}(\tau,\sigma)) =&\frac{2}{T}\sum_{k=1}^{\lfloor T/2 \rfloor}\E\left(p_{\omega_k}^{(T)}(\tau,\sigma)p_{\omega_{k-1}}^{(T)}(\sigma,\tau)\right)\nonumber\\
= &\frac{2}{T}\sum_{k=1}^{\lfloor T/2 \rfloor}\text{Cov}\left(p_{\omega_k}^{(T)}(\tau,\sigma),p_{\omega_{k-1}}^{(T)}(\sigma,\tau)\right) + \frac{2}{T}\sum_{k=1}^{\lfloor T/2 \rfloor}\E\left(p_{\omega_k}^{(T)}(\tau,\sigma)\right)\E\left(p_{\omega_{k-1}}^{(T)}(\sigma,\tau)\right)\nonumber\\
=&\frac{2}{T}\sum_{k=1}^{\lfloor T/2 \rfloor} O(T^{-1}) + \frac{2}{T}\sum_{k=1}^{\lfloor T/2 \rfloor}\left(f_{\omega_k}(\tau,\sigma)f_{\omega_{k-1}}(\sigma,\tau) + O(T^{-1})\right)\nonumber\\
= &\frac{2}{T}\sum_{k=1}^{\lfloor T/2 \rfloor}f_{\omega_k}(\tau,\sigma)f_{\omega_{k-1}}(\sigma,\tau) +O(T^{-1})\nonumber\\
\E(S_{T,1}(\tau,\sigma)) =&\frac{1}{T}\sum_{k=1}^{\lfloor T/2 \rfloor}\E\left(p_{\omega_k}^{(T)}(\tau,\sigma)+p_{\omega_{k}}^{(T)}(\sigma,\tau)\right)
= \frac{1}{T}\sum_{k=1}^{\lfloor T/2 \rfloor} \left(f_{\omega_k}(\tau,\sigma) + f_{\omega_k}(\sigma,\tau)\right) + O(T^{-1}).\nonumber
\end{align}
An upper bound on the approximation error for the integral by sum is given by
\begin{align}
& \Big\vert\frac{1}{T}\sum_{k=1}^{\lfloor T /2\rfloor} f_{\omega_k}(\tau,\sigma)f_{\omega_{k-1}}(\sigma,\tau) - \frac{1}{2\pi}\int_0^{\pi}\vert f_{\omega}(\tau,\sigma)\vert^2d\omega \Big\vert \nonumber\\
\leq &\frac{1}{2\pi}\sum_{k=1}^{\lfloor T /2\rfloor}\sum_{t_1,t_2}\int_{\omega_{k-1}}^{\omega_k}\Big\vert \exp(-i\omega_k t_1 - i\omega_{k-1}t_2) - \exp(-i\omega(t_1 + t_2)) \Big\vert d\omega\vert r_{t_1}(\tau,\sigma)r_{t_2}(\sigma,\tau)\vert \nonumber\\
\leq &\frac{2\pi}{T^2}\sum_{k=1}^{\lfloor T/2 \rfloor}\sum_{t_1,t_2 \leq N}\frac{1}{\vert t_1 + t_2\vert}\vert r_{t_1}(\tau,\sigma)r_{t_2}(\sigma,\tau)\vert  + \sum_{t_1,t_2 > N}\vert r_{t_1}(\tau,\sigma)r_{t_2}(\sigma,\tau)\vert \label{eq:diff1}
\end{align}
As $\sum_{t_1,t_2}\int_0^1\int_0^1\vert r_{t_1}(\tau,\sigma)r_{t_2}(\sigma,\tau)\vert d\tau d\sigma \leq (\sum_t\| r_t \|_2)^2 < \infty$, therefore $\sum_{t_1,t_2}\vert r_{t_1}(\tau,\sigma)r_{t_2}(\sigma,\tau)\vert < \infty$ for almost every $(\tau,\sigma)$, and hence we can choose $N$ appropriately so that the upper bound given in \eqref{eq:diff1} is of order  $O(T^{-1})$.  The term $\E(S_{T,1}(\tau,\sigma)$ can be dealt with similarly.
\end{proof}

\begin{lemma}
\label{Scov}
Under assumption (ii) of Theorem \ref{thm2}, for almost all $(\tau_1,\sigma_1,\tau_2,\sigma_2) \in [0,1]^4$, the limit of the covariance matrix $\Sigma$ of the vector $ \wtilde{I}_T(\tau_1,\tau_2,\sigma_1,\sigma_2)$
defiend in \eqref{itilde} is given by
\begin{align}
\Sigma_{12} = \Sigma_{21} \to & \frac{1}{2\pi}\int_{-\pi}^{\pi}\left(f_{\omega}(\tau_1,\sigma_2)f_{\omega}(\tau_2,\sigma_1) +f_{\omega}(\tau_1,\tau_2)f_{\omega}(\sigma_2,\sigma_1)\right)d\omega \nonumber\\
&+\frac{1}{2\pi}\int_{-\pi}^{\pi} \int_{-\pi}^{\pi} f_{\omega_1,-\omega_1,\omega_2}(\tau_1,\sigma_1,\tau_2,\sigma_2)\omega_1 d\omega_2\nonumber\\
\Sigma_{34} = \Sigma_{43} \to &\frac{2}{\pi}\int_{-\pi}^{\pi}f_{\omega}(\tau_1,\sigma_1)f_{\omega}(\sigma_1,\tau_2)f_{\omega}(\sigma_2,\tau_1)f_{\omega}(\tau_2,\sigma_2)d\omega\nonumber\\ 
&+ \frac{2}{\pi}\int_{-\pi}^{\pi}f_{\omega}(\tau_1,\sigma_1)f_{\omega}(\tau_2,\tau_1)f_{\omega}(\sigma_1,\sigma_2)f_{\omega}(\sigma_2,\tau_2)d\omega\nonumber\\
&+ \frac{1}{\pi}\int_{-\pi}^{\pi}f_{\omega}(\tau_1,\sigma_2)f_{\omega}(\tau_2,\sigma_1)f_{\omega}(\sigma_1,\tau_2)f_{\omega}(\sigma_2,\tau_1)d\omega\nonumber\\
&+\frac{2}{\pi}\int_{-\pi}^{\pi}\int_{-\pi}^{\pi}f_{\omega_1}(\tau_1,\sigma_1)f_{\omega_2}(\tau_2,\sigma_2)f_{\omega_1,-\omega_1,\omega_2}(\sigma_1,\tau_1,\sigma_2,\tau_2)d\omega_1d\omega_2\nonumber\\
\Sigma_{23} = \Sigma_{32} \to &\frac{1}{\pi}\int_{-\pi}^{\pi}f_{\omega}(\tau_1,\sigma_1)f_{\omega}(\sigma_1,\sigma_2)f_{\omega}(\tau_2,\tau_1)d\omega +\frac{1}{\pi}\int_{-\pi}^{\pi}f_{\omega}(\tau_1,\sigma_1)f_{\omega}(\sigma_1,\tau_2)f_{\omega}(\sigma_2,\tau_1)d\omega\nonumber\\
&+\frac{1}{\pi}\int_{-\pi}^{\pi}\int_{-\pi}^{\pi}f_{\omega_1}(\tau_1,\sigma_1)f_{\omega_1,-\omega_1,\omega_2}(\sigma_1,\tau_1,\sigma_2,\tau_2)d\omega_1d\omega_2\nonumber\\
\Sigma_{14} = \Sigma_{41} \to &\frac{1}{\pi}\int_{-\pi}^{\pi}f_{\omega}(\tau_2,\sigma_2)f_{\omega}(\sigma_2,\sigma_1)f_{\omega}(\tau_1,\tau_2)d\omega +\frac{1}{\pi}\int_{-\pi}^{\pi}f_{\omega}(\tau_2,\sigma_2)f_{\omega}(\sigma_2,\tau_1)f_{\omega}(\sigma_1,\tau_2)d\omega\nonumber\\
&+\frac{1}{\pi}\int_{-\pi}^{\pi}\int_{-\pi}^{\pi}f_{\omega_1}(\tau_2,\sigma_2)f_{\omega_1,-\omega_1,\omega_2}(\sigma_2,\tau_2,\sigma_1,\tau_1)d\omega_1d\omega_2 \nonumber\\
\Sigma_{13} = \Sigma_{31} \to &\frac{1}{\pi}\int_{-\pi}^{\pi}f_{\omega}(\tau_1,\sigma_1)f_{\omega}(\sigma_1,\sigma_1)f_{\omega}(\tau_1,\tau_1)d\omega +\frac{1}{\pi}\int_{-\pi}^{\pi}f_{\omega}(\tau_1,\sigma_1)f_{\omega}^2(\sigma_1,\tau_1)d\omega\nonumber\\
&+\frac{1}{\pi}\int_{-\pi}^{\pi}\int_{-\pi}^{\pi}f_{\omega_1}(\tau_1,\sigma_1)f_{\omega_1,-\omega_1,\omega_2}(\sigma_1,\tau_1,\sigma_1,\tau_1)d\omega_1d\omega_2\nonumber\\
\Sigma_{24} = \Sigma_{42} \to &\frac{1}{\pi}\int_{-\pi}^{\pi}f_{\omega}(\tau_2,\sigma_2)f_{\omega}(\sigma_2,\sigma_2)f_{\omega}(\tau_2,\tau_2)d\omega +\frac{1}{\pi}\int_{-\pi}^{\pi}f_{\omega}(\tau_2,\sigma_2)f_{\omega}^2(\sigma_2,\tau_2)d\omega\nonumber\\
&+\frac{1}{\pi}\int_{-\pi}^{\pi}\int_{-\pi}^{\pi}f_{\omega_1}(\tau_2,\sigma_2)f_{\omega_1,-\omega_1,\omega_2}(\sigma_2,\tau_2,\sigma_2,\tau_2)d\omega_1d\omega_2.\nonumber
\end{align}
\begin{align}
\Sigma_{ii} \rightarrow &\ \frac{1}{2\pi}\int_{-\pi}^{\pi}\left(f_{\omega}^2(\tau_i,\sigma_i) +f_{\omega}(\tau_i,\tau_i)f_{\omega}(\sigma_i,\sigma_i)\right)d\omega \nonumber\\
&+\frac{1}{2\pi}\int_{-\pi}^{\pi} \int_{-\pi}^{\pi} f_{\omega_1,-\omega_1,\omega_2}(\tau_i,\sigma_i,\tau_i,\sigma_i)\omega_1 d\omega_2 \hspace{0.1 in} \rm{for} \hspace{0.1 in} i=1,2.\nonumber\\
\Sigma_{ii} \to &\frac{2}{\pi}\int_{-\pi}^{\pi}\left\vert f_{\omega}(\tau_{i-2},\sigma_{i-2})\right\vert^4d\omega + \frac{2}{\pi}\int_{-\pi}^{\pi}\left\vert f_{\omega}(\tau_{i-2},\sigma_{i-2})\right\vert^2f_{\omega}(\tau_{i-2},\tau_{i-2})f_{\omega}(\sigma_{i-2},\sigma_{i-2})d\omega\nonumber\\
&+\frac{2}{\pi}\int_{-\pi}^{\pi}\int_{-\pi}^{\pi}f_{\omega_1}(\tau_{i-2},\sigma_{i-2})f_{\omega_2}(\tau_{i-2},\sigma_{i-2})f_{\omega_1,-\omega_1,\omega_2}(\sigma_{i-2},\tau_{i-2},\sigma_{i-2},\tau_{i-2})d\omega_1d\omega_2,\hspace{0.1 in} \rm{for} \hspace{0.1 in} i=3,4.\nonumber
\end{align}
\end{lemma}

\begin{proof}
We start with
\begin{align}
\Sigma_{12} = \Sigma_{21} =&T\text{Cov}\left(S_{T,1}(\tau_1,\sigma_1),S_{T,1}(\tau_2,\sigma_2)\right)\nonumber\\ 
= & \frac{1}{T}\text{Cov}\left(\sum_{k=1}^{\lfloor T/2\rfloor}\left(p_{\omega_k}^{(T)}(\tau_1,\sigma_1)+p_{\omega_k}^{(T)}(\sigma_1,\tau_1)\right),\sum_{k=1}^{\lfloor T/2\rfloor}\left(p_{\omega_k}^{(T)}(\tau_2,\sigma_2)+p_{\omega_k}^{(T)}(\sigma_2,\tau_2)\right)\right)\nonumber\\
=&\frac{1}{T}\sum_{k=1}^{\lfloor T/2\rfloor}\sum_{l=1}^{\lfloor T/2\rfloor}\text{Cov}\left(p_{\omega_k}^{(T)}(\tau_1,\sigma_1),p_{\omega_l}^{(T)}(\tau_2,\sigma_2)\right) +\frac{1}{T}\sum_{k=1}^{\lfloor T/2\rfloor}\sum_{l=1}^{\lfloor T/2\rfloor}\text{Cov}\left(p_{\omega_k}^{(T)}(\sigma_1,\tau_1),p_{\omega_l}^{(T)}(\sigma_2,\tau_2)\right)\nonumber\\
&+\frac{1}{T}\sum_{k=1}^{\lfloor T/2\rfloor}\sum_{l=1}^{\lfloor T/2\rfloor}\text{Cov}\left(p_{\omega_k}^{(T)}(\tau_1,\sigma_1),p_{\omega_l}^{(T)}(\sigma_2,\tau_2)\right) +\frac{1}{T}\sum_{k=1}^{\lfloor T/2\rfloor}\sum_{l=1}^{\lfloor T/2\rfloor}\text{Cov}\left(p_{\omega_k}^{(T)}(\sigma_1,\tau_1),p_{\omega_l}^{(T)}(\tau_2,\sigma_2)\right)\nonumber
\end{align}
First we calculate
\begin{align}
\text{Cov}\left(p_{\omega_k}^{(T)}(\tau_1,\sigma_1),p_{\omega_l}^{(T)}(\tau_2,\sigma_2)\right) = &\cm(AB,CD)\nonumber\\
=&(A,B,C,D) + (A,C)(B,D) + (A,D)(B,C)\nonumber
\end{align}
where $A= \wtilde{X}_{\omega_k}^{T}(\tau_1)$, $B= \wtilde{X}_{-\omega_k}^{T}(\sigma_1)$, $C= \wtilde{X}_{\omega_l}^{T}(\tau_2)$ and $D= \wtilde{X}_{-\omega_l}^{T}(\sigma_2)$ and the last equality holds by Theorem 2.3.2 from \cite{Brillinger_2001}  and the fact that $(A)=(B)=(C)=(D)=0$.
\begin{align}
(A,B,C,D) = &2\pi T^{-2}(Tf_{\omega_k,-\omega_k,\omega_l}(\tau_1,\sigma_1,\tau_2,\sigma_2)+O(1)\nonumber\\
= &\frac{2\pi}{T}f_{\omega_k,-\omega_k,\omega_l}(\tau_1,\sigma_1,\tau_2,\sigma_2) +O(T^{-2})\nonumber\\
(A,C)(B,D) = &T^{-2}O(1)O(1)=O(T^{-2})\nonumber\\
(A,D)(B,C) = &T^{-2}(T\delta_{k,l}f_{\omega_k}(\tau_1,\sigma_2)+O(1))(T\delta_{k,l}f_{\omega_l}(\tau_2,\sigma_1) + O(1))\nonumber\\
= &\delta_{k,l}f_{\omega_k}(\tau_1,\sigma_2)f_{\omega_l}(\tau_2,\sigma_1) + \delta_{k,l}O(T^{-1}) + O(T^{-2})\nonumber
\end{align}
Therefore if follows
\begin{align}
\Sigma_{12} = \Sigma_{21} =&\frac{1}{T}\sum_{k=1}^{\lfloor T/2\rfloor}f_{\omega_k}(\tau_1,\sigma_2)f_{\omega_k}(\tau_2,\sigma_1) + \frac{1}{T}\sum_{k=1}^{\lfloor T/2\rfloor}f_{\omega_k}(\sigma_1,\tau_2)f_{\omega_k}(\sigma_2,\tau_1)\nonumber\\
&+\frac{1}{T}\sum_{k=1}^{\lfloor T/2\rfloor}f_{\omega_k}(\tau_1,\tau_2)f_{\omega_k}(\sigma_2,\sigma_1) +\frac{1}{T}\sum_{k=1}^{\lfloor T/2\rfloor}f_{\omega_k}(\sigma_1,\sigma_2)f_{\omega_k}(\tau_2,\tau_1)\nonumber\\
&+ \frac{2\pi}{T^2}\sum_{k=1}^{\lfloor T/2\rfloor}\sum_{l=1}^{\lfloor T/2\rfloor}f_{\omega_k,-\omega_k,\omega_l}(\tau_1,\sigma_1,\tau_2,\sigma_2) + \frac{2\pi}{T^2}\sum_{k=1}^{\lfloor T/2\rfloor}\sum_{l=1}^{\lfloor T/2\rfloor}f_{\omega_k,-\omega_k,\omega_l}(\sigma_1,\tau_1,\sigma_2,\tau_2)\nonumber\\
&+ \frac{2\pi}{T^2}\sum_{k=1}^{\lfloor T/2\rfloor}\sum_{l=1}^{\lfloor T/2\rfloor}f_{\omega_k,-\omega_k,\omega_l}(\tau_1,\sigma_1,\sigma_2, \tau_2) + \frac{2\pi}{T^2}\sum_{k=1}^{\lfloor T/2\rfloor}\sum_{l=1}^{\lfloor T/2\rfloor}f_{\omega_k,-\omega_k,\omega_l}(\sigma_1,\tau_1,\tau_2,\sigma_2)\nonumber\\
\to & \frac{1}{2\pi}\int_{-\pi}^{\pi}\left(f_{\omega}(\tau_1,\sigma_2)f_{\omega}(\tau_2,\sigma_1) +f_{\omega}(\tau_1,\tau_2)f_{\omega}(\sigma_2,\sigma_1)\right)d\omega \nonumber\\
&+\frac{1}{2\pi}\int_{-\pi}^{\pi} \int_{-\pi}^{\pi} f_{\omega_1,-\omega_1,\omega_2}(\tau_1,\sigma_1,\tau_2,\sigma_2)\omega_1 d\omega_2 \hspace{0.1 in} \text{as } T \to \infty.\nonumber
\end{align}
Similarly for $i=1,2$, we have
\begin{align}
\Sigma_{ii} \rightarrow &\ \frac{1}{2\pi}\int_{-\pi}^{\pi}\left(f_{\omega}^2(\tau_i,\sigma_i) +f_{\omega}(\tau_i,\tau_i)f_{\omega}(\sigma_i,\sigma_i)\right)d\omega \nonumber\\
&+\frac{1}{2\pi}\int_{-\pi}^{\pi} \int_{-\pi}^{\pi} f_{\omega_1,-\omega_1,\omega_2}(\tau_i,\sigma_i,\tau_i,\sigma_i)\omega_1 d\omega_2 \hspace{0.1 in} \text{as } T \to \infty.\nonumber
\end{align}
Next consider:

\begin{align}
\Sigma_{34} = \Sigma_{43} = &T\text{Cov}\left(S_{T,2}(\tau_1,\sigma_1),S_{T,2}(\tau_2,\sigma_2)\right)\nonumber\\ 
= & \frac{4}{T}\text{Cov}\left(\sum_{k=1}^{\lfloor T/2\rfloor}\left(p_{\omega_k}^{(T)}(\tau_1,\sigma_1)p_{\omega_{k-1}}^{(T)}(\sigma_1,\tau_1)\right),\sum_{k=1}^{\lfloor T/2\rfloor}\left(p_{\omega_k}^{(T)}(\tau_2,\sigma_2)p_{\omega_{k-1}}^{(T)}(\sigma_2,\tau_2)\right)\right)\nonumber\\
= &\frac{4}{T}\sum_{k=1}^{\lfloor T/2\rfloor}\sum_{l=1}^{\lfloor T/2\rfloor}\text{Cov}\left(p_{\omega_k}^{(T)}(\tau_1,\sigma_1)p_{\omega_{k-1}}^{(T)}(\sigma_1,\tau_1),p_{\omega_l}^{(T)}(\tau_2,\sigma_2)p_{\omega_{l-1}}^{(T)}(\sigma_2,\tau_2)\right)\nonumber
\end{align}
We calculate 
\begin{align}
C_{kl}:= &\text{Cov}\left(p_{\omega_k}^{(T)}(\tau_1,\sigma_2)p_{\omega_{k-1}}^{(T)}(\sigma_1,\tau_1),p_{\omega_l}^{(T)}(\tau_2,\sigma_2)p_{\omega_{l-1}}^{(T)}(\sigma_2,\tau_2)\right)\nonumber\\
= &\cm\left(\tilde{X}_{\omega_k}^{(T)}(\tau_1)\tilde{X}_{-\omega_k}^{(T)}(\sigma_1)\tilde{X}_{\omega_{k-1}}^{(T)}(\sigma_1)\tilde{X}_{-\omega_{k-1}}^{(T)}(\tau_1),\tilde{X}_{\omega_l}^{(T)}(\tau_2)\tilde{X}_{-\omega_l}^{(T)}(\sigma_2)\tilde{X}_{\omega_{l-1}}^{(T)}(\sigma_2)\tilde{X}_{-\omega_{l-1}}^{(T)}(\tau_2)\right)\nonumber\\
= &\cm(ABCD,EFGH)\nonumber
\end{align}
We use Theorem 2.3.2 from \cite{Brillinger_2001} to calculate the cumulant. As argued in the proof of Proposition \ref{prop:cum} we only need to look at the partitions $\nu$ with $p=1,2,3,4$.
\\
\underline{p=1}:
$$(A,B,C,D,E,F,G,H) = \frac{(2\pi)^3}{T^{3}}f_{\omega_k,-\omega_k,\omega_{k-1},-\omega_{k-1},\omega_l,-\omega_l,\omega_{l-1}}(\tau_1,\sigma_1,\sigma_1,\tau_1,\tau_2,\sigma_2,\sigma_2,\tau_2) + O(T^{-4})$$
\underline{p=2}:
\begin{align}
(A,B)(C,D,E,F,G,H) = &\frac{(2\pi)^2}{T^{4}}(Tf_{\omega_k}(\tau_1,\sigma_1) + O(1))(Tf_{\omega_{k-1},-\omega_{k-1},\omega_l,-\omega_l,\omega_{l-1}}(\sigma_1,\tau_1,\tau_2,\sigma_2,\sigma_2,\tau_2)+O(1))\nonumber\\
= &\frac{(2\pi)^2}{T^{2}}f_{\omega_k}(\tau_1,\sigma_1)f_{\omega_{k-1},-\omega_{k-1},\omega_l,-\omega_l,\omega_{l-1}}(\sigma_1,\tau_1,\tau_2,\sigma_2,\sigma_2,\tau_2) + O(T^{-3})\nonumber
\end{align}
Similarly all $2+6$ partitions are of order  $ O(T^{-2})$
$$(A,B,C)(D,E,F,G,H) = \frac{(2\pi)^2}{T^{4}} O(1) O(1) = O(T^{-4}).$$
Similarly all $3+5$ partitions are of order$O(T^{-4})$.
$$(A,B,E,F)(C,D,G,H)= \frac{(2\pi)^2}{T^{4}}(Tf_{\omega_k,-\omega_k,\omega_l}(\tau_1,\sigma_1,\tau_2,\sigma_2) + O(1))(Tf_{\omega_{k-1},-\omega_{k-1},\omega_{l-1}}(\sigma_1,\tau_1,\sigma_2,\tau_2) + O(1))$$
Similarly all $4+4$ partitions are of order $  O(T^{-2})$.
\\
\underline{p=3}:
The partitions with significant contributions are:
\begin{align}
(A,B)(C,F)(D,E,G,H) = &T^{-4}(Tf_{\omega_k}(\tau_1,\sigma_1)+O(1))(T\delta_{k-1,l}f_{\omega_{k-1}}(\sigma_1,\sigma_2)+O(1))\nonumber\\
&\ (T\delta_{k-1,l}(2\pi)f_{-\omega_{k-1},\omega_l,\omega_{l-1}}(\tau_1,\tau_2,\sigma_2,\tau_2)+O(1))\nonumber\\
= &\delta_{k-1,l}O(T^{-1}) + O(T^{-3}).\nonumber\\
(A,B)(C,H)(D,E,F,G)= &\delta_{k,l}O(T^{-1}) + O(T^{-3}).\nonumber\\
(A,B)(D,E)(C,F,G,H)= &\delta_{k-1,l}O(T^{-1}) + O(T^{-3}).\nonumber\\
(A,B)(D,G)(C,E,F,H)= &\delta_{k,l}O(T^{-1}) + O(T^{-3}).\nonumber\\
(A,B)(E,F)(C,D,G,H) = &\frac1{T^{4}}(Tf_{\omega_k}(\tau_1,\sigma_1)+O(1))(Tf_{\omega_{l}}(\tau_2,\sigma_2)+O(1))(2\pi Tf_{\omega_{k-1},-\omega_{k-1},\omega_{l-1}}(\sigma_1,\tau_1,\sigma_2,\tau_2)+O(1))\nonumber\\
= &\frac{2\pi}{T}f_{\omega_k}(\tau_1,\sigma_1)f_{\omega_{l}}(\tau_2,\sigma_2)f_{\omega_{k-1},-\omega_{k-1},\omega_{l-1}}(\sigma_1,\tau_1,\sigma_2,\tau_2) + O(T^{-2})\nonumber\\
(A,B)(G,H)(C,D,E,F) = &\frac{2\pi}{T}f_{\omega_k}(\tau_1,\sigma_1)f_{\omega_{l-1}}(\sigma_2,\tau_2)f_{\omega_{k-1},-\omega_{k-1},\omega_{l}}(\sigma_1,\tau_1,\tau_2,\sigma_2) + O(T^{-2})\nonumber\\
(C,D)(E,F)(A,B,G,H)= &\frac{2\pi}{T}f_{\omega_{k-1}}(\sigma_1,\tau_1)f_{\omega_{l}}(\tau_2,\sigma_2)f_{\omega_{k},-\omega_{k},\omega_{l-1}}(\tau_1,\sigma_1, \sigma_2,\tau_2) + O(T^{-2})\nonumber\\
(C,D)(G,H)(A,B,E,F)= &\frac{2\pi}{T}f_{\omega_{k-1}}(\sigma_1,\tau_1)f_{\omega_{l-1}}(\sigma_2,\tau_2)f_{\omega_{k},-\omega_{k},\omega_{l}}(\tau_1,\sigma_1,\tau_2, \sigma_2) + O(T^{-2}).\nonumber
\end{align}
All other terms are of order $ O(T^{-3})$.\\
\underline{p=4}: The partitions with significant contributions are
\begin{align}
(A,B)(C,F)(D,E)(G,H) = &T^{-4}(Tf_{\omega_k}(\tau_1,\sigma_1)+O(1))(T\delta_{k-1,l}f_{\omega_{k-1}}(\sigma_1,\sigma_2)+O(1))\nonumber\\
& \ (T\delta_{k-1,l}f_{\omega_{l}}(\tau_2,\tau_1)+O(1))(Tf_{\omega_{l-1}}(\sigma_2,\tau_2)+O(1))\nonumber\\
= &\delta_{k-1,l}f_{\omega_k}(\tau_1,\sigma_1)f_{\omega_{k-1}}(\sigma_1,\sigma_2)f_{\omega_{l}}(\tau_2,\tau_1)f_{\omega_{l-1}}(\sigma_2,\tau_2) +O(T^{-2})\nonumber\\
(A,B)(C,H)(D,G)(E,F) = &T^{-4}(Tf_{\omega_k}(\tau_1,\sigma_1)+O(1))(T\delta_{k,l}f_{\omega_{k-1}}(\sigma_1,\tau_2)+O(1))\nonumber\\
& \ (T\delta_{k,l}f_{\omega_{l-1}}(\sigma_2,\tau_1)+O(1))(Tf_{\omega_l}(\tau_2,\sigma_2)+O(1))\nonumber\\
= &\delta_{k-1,l}f_{\omega_k}(\tau_1,\sigma_1)f_{\omega_{k-1}}(\sigma_1,\tau_2)f_{\omega_{l-1}}(\sigma_2,\tau_1)f_{\omega_l}(\tau_2,\sigma_2) + O(T^{-2})\nonumber\\
(A,F)(B,E)(C,D)(G,H) = &\delta_{k,l}f_{\omega_k}(\tau_1,\sigma_2)f_{\omega_{l}}(\tau_2,\sigma_1)f_{\omega_{k-1}}(\sigma_1,\tau_1)f_{\omega_{l-1}}(\sigma_2,\tau_2) + O(T^{-2})\nonumber\\
(A,H)(B,G)(C,D)(E,F) = &\delta_{k,l-1}f_{\omega_k}(\tau_1,\tau_2)f_{\omega_{l-1}}(\sigma_2,\sigma_1)f_{\omega_{k-1}}(\sigma_1,\tau_1)f_{\omega_{l}}(\tau_2,\sigma_2) + O(T^{-2})\nonumber\\
(A,F)(B,E)(C,H)(D,G) = &\delta_{k,l}f_{\omega_k}(\tau_1,\sigma_2)f_{\omega_l}(\tau_2,\sigma_1)f_{\omega_{k-1}}(\sigma_1,\tau_2)f_{\omega_{l-1}}(\sigma_2,\tau_1) + O(T^{-2})\nonumber
\end{align}
Contributions of all the other partitions with $p=4$ are $\leq O(T^{-2})$.
Summing up all these terms we get
\begin{align}
\Sigma_{34} = \Sigma_{43} = & \frac{4}{T}\sum_{k=1}^{\lfloor T/2\rfloor}f_{\omega_k}(\tau_1,\sigma_1)f_{\omega_{k-1}}(\sigma_1,\sigma_2)f_{\omega_{k-1}}(\tau_2,\tau_1)f_{\omega_{k-2}}(\sigma_2,\tau_2)\nonumber\\ 
&+\frac{4}{T}\sum_{k=1}^{\lfloor T/2\rfloor}f_{\omega_k}(\tau_1,\sigma_1)f_{\omega_{k-1}}(\sigma_1,\tau_2)f_{\omega_{k-2}}(\sigma_2,\tau_1)f_{\omega_{k-1}}(\tau_2,\sigma_2)\nonumber\\
&+\frac{4}{T}\sum_{k=1}^{\lfloor T/2\rfloor}f_{\omega_k}(\tau_1,\sigma_2)f_{\omega_{k}}(\tau_2,\sigma_1)f_{\omega_{k-1}}(\sigma_1,\tau_1)f_{\omega_{k-1}}(\sigma_2,\tau_2)\nonumber\\ 
&+\frac{4}{T}\sum_{k=1}^{\lfloor T/2\rfloor}f_{\omega_k}(\tau_1,\tau_2)f_{\omega_{k}}(\sigma_2,\sigma_1)f_{\omega_{k-1}}(\sigma_1,\tau_1)f_{\omega_{k+1}}(\tau_2,\sigma_2)\nonumber\\
&+\frac{4}{T}\sum_{k=1}^{\lfloor T/2\rfloor}f_{\omega_k}(\tau_1,\sigma_2)f_{\omega_{k}}(\tau_2,\sigma_1)f_{\omega_{k-1}}(\sigma_1,\tau_2)f_{\omega_{k-1}}(\sigma_2,\tau_1)\nonumber\\
&+\frac{4(2\pi)}{T^2}\sum_{k=1}^{\lfloor T/2\rfloor}\sum_{l=1}^{\lfloor T/2\rfloor}f_{\omega_k}(\tau_1,\sigma_1)f_{\omega_l}(\tau_2,\sigma_2)f_{\omega_{k-1},-\omega_{k-1},\omega_{l-1}}(\sigma_1,\tau_1,\sigma_2,\tau_2)\nonumber\\
&+\frac{4(2\pi)}{T^2}\sum_{k=1}^{\lfloor T/2\rfloor}\sum_{l=1}^{\lfloor T/2\rfloor}f_{\omega_k}(\tau_1,\sigma_1)f_{\omega_{l-1}}(\sigma_2,\tau_2)f_{\omega_{k-1},-\omega_{k-1},\omega_{l}}(\sigma_1,\tau_1,\tau_2,\sigma_2)\nonumber\\
&+\frac{4(2\pi)}{T^2}\sum_{k=1}^{\lfloor T/2\rfloor}\sum_{l=1}^{\lfloor T/2\rfloor}f_{\omega_{k-1}}(\sigma_1,\tau_1)f_{\omega_l}(\tau_2,\sigma_2)f_{\omega_{k},-\omega_{k},\omega_{l-1}}(\tau_1,\sigma_1,\sigma_2,\tau_2)\nonumber\\
&+\frac{4(2\pi)}{T^2}\sum_{k=1}^{\lfloor T/2\rfloor}\sum_{l=1}^{\lfloor T/2\rfloor}f_{\omega_{k-1}}(\sigma_1,\tau_1)f_{\omega_{l-1}}(\sigma_2,\tau_2)f_{\omega_{k},-\omega_{k},\omega_{l}}(\tau_1,\sigma_1,\tau_2,\sigma_2) + o(1)\nonumber\\
\to &\frac{2}{\pi}\int_{-\pi}^{\pi}f_{\omega}(\tau_1,\sigma_1)f_{\omega}(\sigma_1,\tau_2)f_{\omega}(\sigma_2,\tau_1)f_{\omega}(\tau_2,\sigma_2)d\omega\nonumber\\ 
&+ \frac{2}{\pi}\int_{-\pi}^{\pi}f_{\omega}(\tau_1,\sigma_1)f_{\omega}(\tau_2,\tau_1)f_{\omega}(\sigma_1,\sigma_2)f_{\omega}(\sigma_2,\tau_2)d\omega\nonumber\\
&+ \frac{1}{\pi}\int_{-\pi}^{\pi}f_{\omega}(\tau_1,\sigma_2)f_{\omega}(\tau_2,\sigma_1)f_{\omega}(\sigma_1,\tau_2)f_{\omega}(\sigma_2,\tau_1)d\omega\nonumber\\
&+\frac{2}{\pi}\int_{-\pi}^{\pi}\int_{-\pi}^{\pi}f_{\omega_1}(\tau_1,\sigma_1)f_{\omega_2}(\tau_2,\sigma_2)f_{\omega_1,-\omega_1,\omega_2}(\sigma_1,\tau_1,\sigma_2,\tau_2)d\omega_1d\omega_2\nonumber
\end{align}
as $T \to \infty$. Similarly for $i=3,4$,
\begin{align}
\Sigma_{ii} \to &\frac{2}{\pi}\int_{-\pi}^{\pi}\left\vert f_{\omega}(\tau_{i-2},\sigma_{i-2})\right\vert^4d\omega + \frac{2}{\pi}\int_{-\pi}^{\pi}\left\vert f_{\omega}(\tau_{i-2},\sigma_{i-2})\right\vert^2f_{\omega}(\tau_{i-2},\tau_{i-2})f_{\omega}(\sigma_{i-2},\sigma_{i-2})d\omega\nonumber\\
&+\frac{2}{\pi}\int_{-\pi}^{\pi}\int_{-\pi}^{\pi}f_{\omega_1}(\tau_{i-2},\sigma_{i-2})f_{\omega_2}(\tau_{i-2},\sigma_{i-2})f_{\omega_1,-\omega_1,\omega_2}(\sigma_{i-2},\tau_{i-2},\sigma_{i-2},\tau_{i-2})d\omega_1d\omega_2,\nonumber
\end{align}
as $T \to \infty$. Finally we calculate
\begin{align}
\Sigma_{23} = \Sigma_{32} = &T\text{Cov}\left(S_{T,1}(\tau_2,\sigma_2),S_{T2}(\tau_1,\sigma_1)\right)\nonumber\\ 
= &\frac{2}{T}\sum_{k=1}^{\lfloor T/2 \rfloor}\sum_{l=1}^{\lfloor T/2 \rfloor}\text{Cov}\left(p_{\omega_k}^{(T)}(\tau_1,\sigma_1)p_{\omega_{k-1}}^{(T)}(\sigma_1,\tau_1),p_{\omega_l}^{(T)}(\tau_2,\sigma_2)+p_{\omega_l}^{(T)}(\sigma_2,\tau_2)\right).\nonumber
\end{align}
As earlier we consider each of the terms in the summation separately and
calculate
\begin{align}
\text{Cov}\left(p_{\omega_k}^{(T)}(\tau_1,\sigma_1)p_{\omega_{k-1}}^{(T)}(\sigma_1,\tau_1),p_{\omega_l}^{(T)}(\tau_2,\sigma_2)\right) = &\cm(ABCD,EF)\nonumber
\end{align}
We employ Theorem 2.3.2 from \cite{Brillinger_2001} and only calculate cumulants for partitions with size $p=1,2,3$.
\\
For $p=1$: $(A,B,C,D,E,F) = O(T^{-2}).$\\
For $p=2$: all 3+3 partitions = $O(T^{-3})$ and significant 2+4 partitions are as follows:
\begin{align}
(A,B)(C,D,E,F) = &(2\pi)T^{-3}(Tf_{\omega_k}(\tau_1,\sigma_1)+O(1))(Tf_{\omega_{k-1},-\omega_{k-1},\omega_l}(\sigma_1,\tau_1,\tau_2,\sigma_2)+O(1))\nonumber\\
=&\frac{2\pi}{T}f_{\omega_k}(\tau_1,\sigma_1)f_{\omega_{k-1},-\omega_{k-1},\omega_l}(\sigma_1,\tau_1,\tau_2,\sigma_2) + O(T^{-2})\nonumber\\
(A,F)(B,C,D,E) = &(2\pi)T^{-3}(T\delta_{k,l}f_{\omega_k}(\tau_1,\sigma_2)+O(1))(T\delta_{k,l}f_{-\omega_{k},\omega_{k-1},-\omega_{k-l}}(\sigma_1,\sigma_1,\tau_1,\tau_2)+O(1))\nonumber\\
=&\frac{2\pi}{T}\delta_{k,l}f_{\omega_k}(\tau_1,\sigma_2)f_{-\omega_{k},\omega_{k-1},-\omega_{k-l}}(\sigma_1,\sigma_1,\tau_1,\tau_2) + O(T^{-2})\nonumber\\
(B,E)(A,C,D,F) =  &(2\pi)T^{-3}(T\delta_{k,l}f_{-\omega_k}(\sigma_1,\tau_2)+O(1))(T\delta_{k,l}f_{\omega_{k},\omega_{k-1},-\omega_{k-l}}(\tau_1,\sigma_1,\tau_1,\sigma_2)+O(1))\nonumber\\
=&\frac{2\pi}{T}\delta_{k,l}f_{-\omega_k}(\sigma_1,\tau_2)f_{\omega_{k},\omega_{k-1},-\omega_{k-l}}(\tau_1,\sigma_1,\tau_1,\sigma_2) + O(T^{-2})\nonumber\\
(C,D)(A,B,E,F) = &\frac{2\pi}{T}f_{\omega_{k-1}}(\sigma_1,\tau_1)f_{\omega_{k},-\omega_{k},\omega_l}(\tau_1,\sigma_1,\tau_2,\sigma_2) + O(T^{-2})\nonumber\\
(C,F)(A,B,D,F) = &\frac{2\pi}{T}\delta_{k-1,l}f_{\omega_{k-1}}(\sigma_1,\sigma_2)f_{\omega_{k},-\omega_{k},\omega_{l}}(\tau_1,\sigma_1,\tau_2,\tau_1) + O(T^{-2})\nonumber\\
(D,E)(A,B,C,F) = &\delta_{k-1,l}O(T^{-1})+O(T^{-2}).\nonumber
\end{align} 
Other 2+4 partitions are $O(T^{-3})$.\\
For p=3:
\begin{align}
(A,B)(C,F)(D,E) =&T^{-3}(Tf_{\omega_k}(\tau_1,\sigma_1)+O(1))(T\delta_{k-1,l}f_{\omega_{k-1}}(\sigma_1,\sigma_2)+O(1))(T\delta_{k-1,l}f_{\omega_{l}}(\tau_2,\tau_1)+O(1)) \nonumber\\
= &\delta_{k-1,l}f_{\omega_k}(\tau_1,\sigma_1)f_{\omega_{k-1}}(\sigma_1,\sigma_2) f_{\omega_{l}}(\tau_2,\tau_1)+ \delta_{k-1,l}O(T^{-1}) + O(T^{-2})\nonumber\\
(A,F)(B,E)(C,D) = &\delta_{k,l}f_{\omega_k}(\tau_1,\sigma_2)f_{\omega_{l}}(\tau_2,\sigma_1)f_{\omega_{k-1}}(\sigma_1,\tau_1) + \delta_{k,l}O(T^{-1}) + O(T^{-2})\nonumber
\end{align}
and other partitions are $O(T^{-2})$.
Summing up all the cumulants we obtain 
\begin{align}
\Sigma_{23} = \Sigma_{32} = &\frac{2}{T}\sum_{k=1}^{\lfloor T/2 \rfloor}\left(f_{\omega_k}(\tau_1,\sigma_1)f_{\omega_{k-1}}(\sigma_1,\sigma_2) f_{\omega_{k-1}}(\tau_2,\tau_1)+f_{\omega_k}(\tau_1,\sigma_1)f_{\omega_{k-1}}(\sigma_1,\tau_2) f_{\omega_{k-1}}(\sigma_2,\tau_1)\right)\nonumber\\
&+\frac{2}{T}\sum_{k=1}^{\lfloor T/2 \rfloor}\left(f_{\omega_k}(\tau_1,\sigma_2)f_{\omega_{k}}(\tau_2,\sigma_1)f_{\omega_{k-1}}(\sigma_1,\tau_1)+f_{\omega_k}(\tau_1,\tau_2)f_{\omega_{k}}(\sigma_2,\sigma_1)f_{\omega_{k-1}}(\sigma_1,\tau_1)\right)\nonumber\\
&+ \frac{4\pi}{T^2}\sum_{k=1}^{\lfloor T/2 \rfloor}\sum_{l=1}^{\lfloor T/2 \rfloor}f_{\omega_k}(\tau_1,\sigma_1)f_{\omega_{k-1},-\omega_{k-1}\omega_l}(\sigma_1,\tau_1,\tau_2,\sigma_2)\nonumber\\ 
&+ \frac{4\pi}{T^2}\sum_{k=1}^{\lfloor T/2 \rfloor}\sum_{l=1}^{\lfloor T/2 \rfloor}f_{\omega_k}(\tau_1,\sigma_1)f_{\omega_{k-1},-\omega_{k-1}\omega_l}(\sigma_1,\tau_1,\sigma_2, \tau_2)\nonumber\\
&+ \frac{4\pi}{T^2}\sum_{k=1}^{\lfloor T/2 \rfloor}\sum_{l=1}^{\lfloor T/2 \rfloor}f_{\omega_{k-1}}(\sigma_1, \tau_1)f_{\omega_{k},-\omega_{k}\omega_l}(\tau_1,\sigma_1,\tau_2,\sigma_2)\nonumber\\ 
&+ \frac{4\pi}{T^2}\sum_{k=1}^{\lfloor T/2 \rfloor}\sum_{l=1}^{\lfloor T/2 \rfloor} f_{\omega_{k-1}}(\sigma_1,\tau_1)f_{\omega_{k},-\omega_{k}\omega_l}(\tau_1,\sigma_1,\sigma_2, \tau_2) + o(1)\nonumber\\
\to &\frac{1}{\pi}\int_{-\pi}^{\pi}f_{\omega}(\tau_1,\sigma_1)f_{\omega}(\sigma_1,\sigma_2)f_{\omega}(\tau_2,\tau_1)d\omega +\frac{1}{\pi}\int_{-\pi}^{\pi}f_{\omega}(\tau_1,\sigma_1)f_{\omega}(\sigma_1,\tau_2)f_{\omega}(\sigma_2,\tau_1)d\omega\nonumber\\
&+\frac{1}{\pi}\int_{-\pi}^{\pi}\int_{-\pi}^{\pi}f_{\omega_1}(\tau_1,\sigma_1)f_{\omega_1,-\omega_1,\omega_2}(\sigma_1,\tau_1,\sigma_2,\tau_2)d\omega_1d\omega_2\nonumber \hspace{0.1 in} \text{as } T \to \infty.
\end{align}
The other terms are obtained similarly.
\end{proof}

\begin{prop}\label{ITmean}
Under  assumption~\eqref{cond:mixing} of \autoref{thm2}, for almost every $(\tau,\sigma) \in [0,1]^2$,
$$\E(S_{T2}(\tau,\sigma)-S_{T1}(\tau,\sigma)\overline{S_{T1}(\tau,\sigma)}) = \frac{1}{2\pi}\int_{-\pi}^{\pi}\vert f_{\omega}(\tau,\sigma)\vert^2d\omega - \frac{1}{4\pi^2}\left\vert\int_{-\pi}^{\pi}f_{\omega}(\tau,\sigma)d\omega\right\vert^2 + O(T^{-1}).$$ 
\end{prop}

\begin{proof}
The result follows from Lemma \ref{Smean}, Lemma \ref{Scov} and the fact
$\E\left(S_{T1}(\tau,\sigma)\overline{S_{T1}(\tau,\sigma)}\right) =\text{Var}(S_{T1}(\tau,\sigma)) + \E^2(S_{T,1}(\tau,\sigma))$.
\end{proof}

\begin{prop}
\label{ITcov}
Under assumption (ii) of Theorem \ref{thm2}, for almost every $(\tau_1,\sigma_1,\tau_2,\sigma_2) \in [0,1]^4$,
\begin{align}
\nu^2(\tau_1,\sigma_1,\tau_2,\sigma_2) &=\lim_{T \to \infty}\operatorname{Cov}(I_T(\tau_1,\sigma_1),I_T(\tau_2,\sigma_2))\label{cov_ker_IT} \\
&= \frac{2}{\pi}\int_{-\pi}^{\pi}f_{\omega}(\tau_1,\sigma_1)f_{\omega}(\sigma_1,\tau_2)f_{\omega}(\sigma_2,\tau_1)f_{\omega}(\tau_2,\sigma_2)d\omega\nonumber\\ 
&+ \frac{2}{\pi}\int_{-\pi}^{\pi}f_{\omega}(\tau_1,\sigma_1)f_{\omega}(\tau_2,\tau_1)f_{\omega}(\sigma_1,\sigma_2)f_{\omega}(\sigma_2,\tau_2)d\omega\nonumber\\
&+ \frac{1}{\pi}\int_{-\pi}^{\pi}f_{\omega}(\tau_1,\sigma_2)f_{\omega}(\tau_2,\sigma_1)f_{\omega}(\sigma_1,\tau_2)f_{\omega}(\sigma_2,\tau_1)d\omega\nonumber\\
&+\frac{2}{\pi}\int_{-\pi}^{\pi}\int_{-\pi}^{\pi}f_{\omega_1}(\tau_1,\sigma_1)f_{\omega_2}(\tau_2,\sigma_2)f_{\omega_1,-\omega_1,\omega_2}(\sigma_1,\tau_1,\sigma_2,\tau_2)d\omega_1d\omega_2\nonumber\\
&-\frac{1}{\pi^2}\int_{-\pi}^{\pi}f_{\omega}(\tau_1,\sigma_1)d\omega\int_{-\pi}^{\pi}f_{\omega}(\tau_2,\sigma_2)f_{\omega}(\sigma_2,\sigma_1)f_{\omega}(\tau_1,\tau_2)d\omega\nonumber\\ 
&-\frac{1}{\pi^2}\int_{-\pi}^{\pi}f_{\omega}(\tau_1,\sigma_1)d\omega\int_{-\pi}^{\pi}f_{\omega}(\tau_2,\sigma_2)f_{\omega}(\sigma_2,\tau_1)f_{\omega}(\sigma_1,\tau_2)d\omega\nonumber\\
&-\frac{1}{\pi^2}\int_{-\pi}^{\pi}f_{\omega}(\tau_1,\sigma_1)d\omega\int_{-\pi}^{\pi}\int_{-\pi}^{\pi}f_{\omega_1}(\tau_2,\sigma_2)f_{\omega_1,-\omega_1,\omega_2}(\sigma_2,\tau_2,\sigma_1,\tau_1)d\omega_1d\omega_2\nonumber\\
&-\frac{1}{\pi^2}\int_{-\pi}^{\pi}f_{\omega}(\tau_2,\sigma_2)d\omega\int_{-\pi}^{\pi}f_{\omega}(\tau_1,\sigma_1)f_{\omega}(\sigma_1,\sigma_2)f_{\omega}(\tau_2,\tau_1)d\omega\nonumber\\ 
&-\frac{1}{\pi^2}\int_{-\pi}^{\pi}f_{\omega}(\tau_2,\sigma_2)d\omega\int_{-\pi}^{\pi}f_{\omega}(\tau_1,\sigma_1)f_{\omega}(\sigma_1,\tau_2)f_{\omega}(\sigma_2,\tau_1)d\omega\nonumber\\
&-\frac{1}{\pi^2}\int_{-\pi}^{\pi}f_{\omega}(\tau_2,\sigma_2)d\omega\int_{-\pi}^{\pi}\int_{-\pi}^{\pi}f_{\omega_1}(\tau_1,\sigma_1)f_{\omega_1,-\omega_1,\omega_2}(\sigma_1,\tau_1,\sigma_2,\tau_2)d\omega_1d\omega_2\nonumber\\
&+ \frac{1}{2\pi^3}\int_{-\pi}^{\pi}f_{\omega}(\tau_1,\sigma_1)d\omega\int_{-\pi}^{\pi}f_{\omega}(\tau_2,\sigma_2)d\omega\int_{-\pi}^{\pi}f_{\omega}(\tau_1,\sigma_2)f_{\omega}(\tau_2,\sigma_1)d\omega\nonumber\\ 
&+\frac{1}{2\pi^3}\int_{-\pi}^{\pi}f_{\omega}(\tau_1,\sigma_1)d\omega\int_{-\pi}^{\pi}f_{\omega}(\tau_2,\sigma_2)d\omega\int_{-\pi}^{\pi}f_{\omega}(\tau_1,\tau_2)f_{\omega}(\sigma_2,\sigma_1)d\omega \nonumber\\
&+\frac{1}{2\pi^3}\int_{-\pi}^{\pi}f_{\omega}(\tau_1,\sigma_1)d\omega\int_{-\pi}^{\pi}f_{\omega}(\tau_2,\sigma_2)d\omega\int_{-\pi}^{\pi} \int_{-\pi}^{\pi} f_{\omega_1,-\omega_1,\omega_2}(\tau_1,\sigma_1,\tau_2,\sigma_2)\omega_1 d\omega_2 \nonumber
\end{align}
\end{prop}

\begin{proof}
We have proved in Section \ref{subsec:fdds} the vector
$$\sqrt{T}\left(\begin{array}{c}
S_{T,1}(\tau_1,\sigma_1) - \E(S_{T,1}(\tau_1,\sigma_1))\\
S_{T,1}(\tau_2,\sigma_2) - \E(S_{T,1}(\tau_2,\sigma_2))\\
S_{T,2}(\tau_1,\sigma_1) - \E(S_{T,2}(\tau_1,\sigma_1))\\
S_{T,2}(\tau_2,\sigma_2) - \E(S_{T,2}(\tau_2,\sigma_2))
\end{array}\right)$$
converges in distribution to a normal distribution. To obtain the covariance kernel of $I_{T}(\tau,\sigma)$, we use delta-method on this vector with $g(x_1,x_2,x_3,x_4) := (x_3-x_1^2,x_4-x_2^2).$ Using Lemma \ref{Smean} 
it follows that
\begin{eqnarray*}
&&\sqrt{T}\left(\begin{array}{c}
S_{T2}(\tau_1,\sigma_1)-S_{T1}(\tau_1,\sigma_1)\overline{S_{T1}(\tau_1,\sigma_1)} - \frac{1}{2\pi}\int_{-\pi}^{\pi}\vert f_{\omega}(\tau_1,\sigma_1)\vert^2d\omega + \frac{1}{4\pi^2}\left\vert\int_{-\pi}^{\pi}f_{\omega}(\tau_1,\sigma_1)d\omega\right\vert^2\\
S_{T2}(\tau_2,\sigma_2)-S_{T1}(\tau_2,\sigma_2)\overline{S_{T1}(\tau_2,\sigma_2)} - \frac{1}{2\pi}\int_{-\pi}^{\pi}\vert f_{\omega}(\tau_2,\sigma_2)\vert^2d\omega + \frac{1}{4\pi^2}\left\vert\int_{-\pi}^{\pi}f_{\omega}(\tau_2,\sigma_2)d\omega\right\vert^2
\end{array}
\right) \\
&& ~~~~~~~~~~~~~~~~~~~~~~~~~~~~~~~~~~~~~~~~~~~~~~~~~~~~~\stackrel{d}{\to} N(0,\wtilde{\Sigma}).
\end{eqnarray*}
where 
\begin{align}
\wtilde{\Sigma}_{ij} =&\Sigma_{(i+2)(j+2)} -2\frac{1}{2\pi}\int_{-\pi}^{\pi}f_{\omega}(\tau_i,\sigma_i)d\omega\Sigma_{i(j+2)} -2\frac{1}{2\pi}\int_{-\pi}^{\pi}f_{\omega}(\tau_j,\sigma_j)d\omega\Sigma_{(i+2)j} \nonumber\\ 
&+ 4\frac{1}{4\pi^2}\int_{-\pi}^{\pi}f_{\omega}(\tau_i,\sigma_i)d\omega\int_{-\pi}^{\pi}f_{\omega}(\tau_j,\sigma_j)d\omega\Sigma_{ij}\nonumber
\end{align}
 for $i=1,2$; $j=1,2$ and $\Sigma_{ij}$ are as in Lemma \ref{Scov}. 
Substituting the values of $\Sigma_{ij}$ we obtain \eqref{cov_ker_IT}.
\end{proof}

\subsection{The order of $C_0(\nu)$ in Proposition \ref{prop:cum}}
\label{subsec:prop_cm}
Using the notations of Proposition \ref{prop:cum} we write 
\begin{equation}
\label{cum_v}
C_0(\nu) = \frac{2^l}{T^{l/2}}\sum_{k_1=1}^{\lfloor T/2\rfloor} \dots\sum_{k_l=1}^{\lfloor T/2\rfloor} \cm\left(Z_{ij} : ij \in \nu_1\right) \dots \cm\left(Z_{ij} : ij \in \nu_p\right),
\end{equation}
where
$Z_{i1}:= \wtilde{X}_{\omega_i}^{(T)}(\tau), \ 
Z_{i2}:=  \wtilde{X}_{-\omega_i}^{(T)}(\sigma), \ 
Z_{i3}:=  \wtilde{X}_{\omega_{i-1}}^{(T)}(\sigma) \ \text{and} \
Z_{i4}:=  \wtilde{X}_{-\omega_{i-1}}^{(T)}(\tau)$
and  $\nu = \nu_1 \cup \nu_2 \cup \dots \cup \nu_p$ is any indecomposable partition of the table
\begin{equation}
\label{table_index}
\begin{array}{cccc}
(k_1,1) & (k_1,2) & (k_1,3) & (k_1,4) \\
(k_2,1) & (k_2,2) & (k_2,3) & (k_2,4) \\
\vdots &\vdots &\vdots &\vdots \\
(k_l,1) & (k_l,2) & (k_l,3) & (k_l,4) 
\end{array}.
\end{equation}

To calculate these cumulants we will use Theorem B.2 from \cite{Panaretos_Tavakoli_2013_Supp}, which says
\begin{equation}
\label{Thm_B2}
\cm\left( \wtilde{X}_{\omega_1}^{(T)}(\tau_1), \dots , \wtilde{X}_{\omega_k}^{(T)}(\tau_k)\right) = \frac{(2\pi)^{k/2 - 1}}{T^{k/2}} \Delta^{(T)}(\omega_1 + \dots, \omega_k)f_{\omega_1, \dots, \omega_{k-1}}(\tau_1,\dots,\tau_k) + \epsilon_T.
\end{equation}
In the above equation $f_{\omega_1, \dots, \omega_{k-1}} = O(1)$, $\epsilon_T = O(1)$ uniformly over $\omega$ and the equality is interpreted in $L^2([0,1]^k)$.
\\
For $\omega = 2\pi k /T,\  k \in \mathbb{Z}$, the function $\Delta^{(T)}(\omega)=T$ if $k = 0 \ (mod\  T)$ and $0$ otherwise. 
\\
Note that if for some $\nu$, we have $p > 2l$ then there is at least one $\vert \nu_m \vert = 1$ for some $1 \leq m \leq p$. In that case $\cm(Z_{ij} :ij \in \nu_m) = 0$ and therefore $C(\nu)=0$.
\\
Therefore let us look at the indecomposable partitions $\nu$ with $p \leq 2l$, such that each $\nu_i$ has at least 2 elements.
\\
Let $\mathcal{P}$ be the set of all partitions of set $\{k_1,k_2,\dots,k_l\}$ and define a function $s : \mathbb{N}^l \times \mathcal{P} \mapsto \{0,1\}^p$, such that $[s(\{k_1,\dots,k_l\},\nu)]_i := 0$ if $\sum \omega_k = 0$ in $\nu_i$ and it is $1$ otherwise. For a fixed partition  $\nu = \nu_1 \cup \dots \cup \nu_p$, using \eqref{Thm_B2},the sum in \eqref{cum_v}  can be written as
\begin{eqnarray}
C(\nu) = \frac{2^l}{T^{l/2}}\sum_{s(\nu) \in \{0,1\}^p}\sum_{\substack{\{k_1,\dots,k_l\}:\\ s(\{k_1,\dots,k_l\},\nu)=s(\nu)}}T^{-2l} T^{p-\| s(\nu) \|} O(1).
\end{eqnarray}
For every possible value of $s(\nu)$, we will find $r$, possible order for the set $\{k_1,\dots,k_l\}$ and an upper bound for $r + p- \| s(\nu) \|$. Note that for some values of $s(\nu)$ there are no feasible solutions for $\{k_1,\dots,k_l\}$ and hence the contribution of such $s(\nu)$ in the sum will be $0$. So we focus on consistent values of $s(\nu)$.

To this end let $s(\nu_j) = [s(\nu)]_j$ and partition table \eqref{table_index} in blocks in the following way. First we look at the rows for which there is no $\nu_j$ such that
$\nu_j \cap \{k_i,-k_i,k_i-1,-(k_i-1)\} \neq \emptyset$ and $s(\nu_j) = 0$. In other words if any set $\nu_j$ in the partition has an elements from $i$-th row then $s(\nu_j) =1$. Each of these rows are one of the blocks, call them $B_{11}, \dots, B_{1r_1}$. On the rest of the rows define the following equivalence relationship. We say $i \sim j$ if there is a chain of sets $\nu_{m_1}, \dots, \nu_{m_t}$ connecting $i$-th and $j$-th row, such that $s(\nu_{m_k}) =0$ for all $k$. It is easy to see it is in fact an equivalence relation. Therefore consider all the equivalence classes and that will give us a partition of the rows of the table. Each of these partitions are considered as separate blocks. Note that by construction each row in one block has a linear relationship with all the other rows in the same block.
Reorder and label the blocks as $B_{21}, \dots , B_{2r_2}, \dots, B_{l1}, \dots, B_{lr_l}$ such that $B_{ij}$ has $(i-1)$ independent solutions for the rows of $B_{ij}$, in the sense that, if we fix any $i-1$ rows of the block the rest will be fixed. Also by construction, if a set $\nu_i$ has an element from both $B_{i_1j_1}$ and $B_{i_2j_2}$, with $(i_1,j_1) \neq (i_2,j_2)$, then $s(\nu_i) =1$.

\begin{claim}
\label{max_set}
For all blocks $B_{ij}$ with $i>2$, there exists sets $\nu_{m_1}, \dots, \nu_{m_t}$ with the property that $s(\nu_{m_k}) =0$ and $\vert \nu_{m_k} \vert >2$ for $k=1, \dots t$; $\cup \nu_{m_j} \subset B_{ij}$ and $\vert \cup \nu_{m_k} \vert \geq i + 2(t -1)$.
\end{claim}

\begin{proof}
By construction we can always find $\nu_{m_1}, \dots, \nu_{m_t}$ such that $s(\nu_{m_k}) =0$ and $\cup \nu_{m_j} \subset B_{ij}$. First suppose $\vert \nu_{m_k} \vert =2$ for $k=1, \dots t$. Note that by construction all the rows in the block must have one element in the $\nu_{m_k}$'s. Therefore all the rows are linearly related and if we fix one row, the rest of the rows is also  fixed. This is a contradiction to the property that $B_{ij}$ has $(i-1)$ independent solutions for $\{k_1,\dots,k_{l}\}$. Therefore there must be at least one set with cardinality $>2$ among the $\nu_{m_k}$'s. 

Now look at the set, $\nu_{m_1}$ with maximal cardinality. If $\vert \nu_{m_1} \vert \geq i$, then the claim holds with $t=1$. If not,consider the rows that do not occur in $\nu_{m_1}$. By construction, all these rows must hook with all the rows appearing in $\nu_{m_1}$ with sets $\nu_k$ such that $s(\nu_k) =0$. Find $\nu_{m_2}$ such that $s(\nu_{m_2})=0$ and $\nu_{m_2}$ has at least one elements from rows appearing in $\nu_{m_1}$ and at least two rows from the rows that do not appear in $\nu_{m_1}$. There must be one such set, because if not once the rows appearing in $\nu_{m_1}$ are fixed, the  rest of the rows in the block is also fixed and in this situation number of independent rows from $B_{ij}$ is $\leq \vert \nu_{m_1} \vert -1 < i-1$. Continue in this way to find $\nu_{m_1}, \dots \nu_{m_t}$. Each of these sets adds at most $\vert\nu_{m_k}\vert -2$ independent variables, therefore number of independent rows in the block $$i-1 \leq \vert \nu_{m_1} \vert -1 + \sum_{k=2}^t (\vert \nu_{m_k} \vert - 2) = \sum_{k=1}^t \vert \nu_{m_k} \vert - 1 - 2(t-1).$$ Consequently, we have 
$i + 2(t -1) \leq  \sum_{k=1}^t \vert \nu_{m_k} \vert =  \vert \cup \nu_{m_k} \vert.$
\end{proof}

Note that each of the first $r_1$ rows can be chosen independently. Therefore from the construction $r = r_1 + r_2 + 2r_3 + \dots (m-1)r_m \leq l$.

Define the set $I := \{i \in \{1,\dots,p\}: s(\nu_i)=0\}$. Note that $\vert I \vert = p - \|s(\nu)\|$. First we find $t_{ij}$ sets from $B_{ij}$ obtained by Claim \ref{max_set}. Let $n$ be the number of elements left in the table, then $n \leq 4l- \sum_{i=3}^m ir_i -2\sum_{i=3}^m \sum_{j=1}^{r_i} (t_{ij}-1)$.

Next we find a lower bound on number of elements in the set $\cup_{k\in I^c} \nu_k$. Because all the sets of the partition must have at least two elements, we have 
\begin{align}
\vert I \vert \leq &\sum_{i=3}^m \sum_{j=1}^{r_i}t_{ij} + \frac{n-\vert \cup_{k\in I^c} \nu_k \vert}{2}
\leq \sum_{i=3}^m r_i + 2l - \frac{\sum_{i=3}^m ir_i + \vert \cup_{k\in I^c} \nu_k \vert}{2}
\end{align}
To this end we consider the following cases separately:
\\
\underline{Case I:} There are more than one blocks, i.e., $\sum r_i > 1$.
\begin{claim}
Each of the blocks must have at least 2 elements from the set $\cup_{i\in I^c} \nu_i$. 
\end{claim} 
\begin{proof}
All the blocks must communicate, therefore there must be at least one element in each block which belongs to some $\nu_j$ that connects two different blocks and hence $s(\nu_j) = 1$. If there are more than one such elements we are done. 
If not, then there must be at least one other set $\nu_j$ consisting of only elements in that block with $s(\nu_j)=1$, owing to the fact that sum of all elements in one block is 0 and it must be nonzero if we take just one element out from the block.
\end{proof}

Therefore in this case each of the blocks must have at least 2 elements in the set $\cup_{i\in I^c} \nu_i$ and hence $\vert \cup_{i\in I^c} \nu_i \vert \geq 2(r_1+r_2+\dots+r_m)$ and $$\vert I \vert \leq 2l - r_1 - r_2 - \dots - \frac{mr_m}{2}.$$

Consequently, it follows that
\begin{align}
r +p-\| s(\nu) \| =  r + \vert I \vert &\leq 2l+ r_3/2 + \dots + (m/2-1)r_m\nonumber\\
&< 2l + r_1 + r_2 + \dots [(m-1)/2]r_m \leq 2l+l/2.\nonumber
\end{align}
The second inequality is strict because of the fact that at least one of the $r_i$'s is positive.
Therefore in this case $C(\nu) =O(T^{\delta})$ for some $\delta < 0$.
\\
\underline{Case II:} There is only one block, i.e., $r_k =1$ for some $2 \leq k \leq m$ and $r_j =0$ for all $j \neq k$.\\
Case II.1: $l=2$: In this situation $k = 2$, as $k \leq l$. Therefore substituting $r=1$ and $p \leq 2l$, we obtain
$$C(\nu) \leq T^{-l/2} O(T) T^{-2l} T^{2l} O(1) = O(T^{1-l/2}) = O(1).$$ 
Case II.2: $l >2$: Here $k \leq l$ and $r=k-1$. By Claim \ref{max_set}, the  total number of sets in the partition $p \leq t + (4l-k-2(t-1))/2 = 2l -k/2 +1/2$. Finally, as $\| s(\nu) \| \geq 0$, we get
$$C(\nu) \leq T^{-l/2} O(T^{k-1}) T^{-2l} T^{2l-k/2+1/2} O(1) = O(T^{k/2-l/2-1/2}) \leq O(T^{-1/2}).$$
This implies that the 2nd order cumulant is finite and cumulants of order $>2$ converges to 0 as $T \to \infty$.  

\end{document}